\documentclass[reqno]{amsart}

\newcommand{\BlackBoxes}{\global\overfullrule5pt}

\BlackBoxes
\newcommand\one{{\bf 1}}

\newcommand\eid{\stackrel{\rm d}{=}}

 \usepackage[dvips]{epsfig}  
 \usepackage{mathrsfs}
\usepackage{amssymb, amsmath}  
\usepackage[mathcal]{euscript}  
\usepackage{eufrak}  
\usepackage{natbib}   

\numberwithin{equation}{section}   
   
\theoremstyle{plain}   
\newtheorem{thm}{Theorem}[section]   
\newtheorem{lem}{Lemma}[section]   
\newtheorem{cor}{Corollary}[section]   
\newtheorem{prop}{Proposition}[section]   
   
\theoremstyle{definition}   
   
\newtheorem{exmp}{Example}[section]   
\theoremstyle{remark}   
\newtheorem{rem}{Remark}[section]

\newcommand{\ind}{I\!}
\newcommand{\eqdis}{\stackrel{\lower0.2ex\hbox{$\scriptscriptstyle   
                    \mathrm{d}$}}{=}}   
\newcommand{\vague}{\stackrel{\lower0.2ex\hbox{$\scriptscriptstyle   
                    \it{v} $}}{\rightarrow}}   
\newcommand{\weak}{\stackrel{\lower0.2ex\hbox{$\scriptscriptstyle   
                    \it{w} $}}{\rightarrow}}   
\newcommand{\what}{\stackrel{\lower0.2ex\hbox{$\scriptscriptstyle   
                    \it{\hat{w}} $}}{\rightarrow}}   
\newcommand{\distr}{\stackrel{\lower0.2ex\hbox{$\scriptscriptstyle   
                    \it{d} $}}{\rightarrow}}   
\newcommand{\ntoinf}{\stackrel{\lower0.2ex\hbox{$\scriptscriptstyle   
                    \it{n\to\infty} $}}{\rightarrow}}   
\newcommand{\as}{\stackrel{\lower0.2ex\hbox{$\scriptscriptstyle   
                    \it{a.s.} $}}{\rightarrow}}   
\newcommand{\inprob}{\stackrel{\lower0.2ex\hbox{$\scriptscriptstyle   
      \Prob$}}{\rightarrow}}

\newcommand{\Eb}{\mathbf{E}}
\newcommand{\Prob}{{P}}
\newcommand{\Var}{\operatorname{Var}}   

\newcommand{\R}{\mathbf{R}}   
\newcommand{\N}{\mathbf{N}}   

\newcommand{\Z}{\mathbf{Z}} 

\newcommand{\Rdbar}{\overline{\mathbb{R}}^d}  

\newcommand{\Db}{\mathbf{D}}
\newcommand{\bfM}{\mathbf{M}}

\newcommand{\ud}{\mathrm{d}}   
\newcommand{\E}{{E}}

\newcommand{\zerobold}{{\mathbf{0}}}

\newcommand{\zbold}{{\mathbf{z}}}   
\newcommand{\xbold}{{\mathbf{x}}}   
\newcommand{\ybold}{{\mathbf{y}}}   
\newcommand{\Abold}{{\mathbb{A}}}   
\newcommand{\Bbold}{{\mathbf{B}}}   
   
\newcommand{\Zbold}{{\mathbf{Z}}}

\newcommand{\cbold}{{\mathbf{c}}}   
\newcommand{\ebold}{{\mathbf{e}}}

\newcommand{\Thetabold}{{\boldsymbol{\Theta}}} 
\newcommand{\Leb}{\operatorname{Leb}}

\newcommand{\RV}{{\rm RV}}

\newcommand{\bth}{\begin{theorem}}   
\newcommand{\ethe}{\end{theorem}}   
   
\newcommand{\bre}{\begin{remark}\em }   
\newcommand{\ere}{\end{remark}}   
   
\newcommand{\ble}{\begin{lemma}}   
\newcommand{\ele}{\end{lemma}}

\newcommand{\bde}{\begin{definition}}   
\newcommand{\ede}{\end{definition}}   
   
\newcommand{\bco}{\begin{corollary}}   
\newcommand{\eco}{\end{corollary}}   
   
\newcommand{\bpr}{\begin{proposition}}   
\newcommand{\epr}{\end{proposition}}   
   
\newcommand{\bexer}{\begin{exercise}}   
\newcommand{\eexer}{\end{exercise}}   
   
\newcommand{\bexam}{\begin{example}}   
\newcommand{\eexam}{\end{example}}   
   
\newcommand{\bfi}{\begin{fig}}   
\newcommand{\efi}{\end{fig}}   
   
\newcommand{\btab}{\begin{tab}}   
\newcommand{\etab}{\end{tab}}   

\newcommand{\beao}{\begin{eqnarray*}}   
\newcommand{\eeao}{\end{eqnarray*}\noindent}   
   
\newcommand{\beam}{\begin{eqnarray}}   
\newcommand{\eeam}{\end{eqnarray}\noindent}   
   
\newcommand{\beqq}{\begin{equation}}   
\newcommand{\eeqq}{\end{equation}\noindent}   
   
\newcommand{\bce}{\begin{center}}   
\newcommand{\ece}{\end{center}}   
   
\newcommand{\barr}{\begin{array}}   
\newcommand{\earr}{\end{array}}

\newcommand{\bdis}{\begin{displaymath}}   
\newcommand{\edis}{\end{displaymath}\noindent}

\newcommand{\vep}{\varepsilon}

\newcommand{\bbr}{{\mathbb R}}

\newcommand{\bbz}{{\mathbb Z}}

\newcommand{\cadlag}{c\`adl\`ag}

\newcommand{\bfZ}{{\bf Z}} 
   
\newcommand{\bfS}{{\bf S}}

\newcommand{\bfc}{{\bf c}}

\newcommand{\Np}{\mathbf{N}_p}
\newcommand{\Nps}{\mathbf{N}_{p,0}}   

\allowdisplaybreaks

\makeatletter
\AtBeginDocument{%
 \def\@serieslogo{%
 \vbox to\headheight{%
 \parindent\z@ \fontsize{6}{7\p@}\selectfont
\endgraf
 \vss}}}
\makeatother

\copyrightinfo{the authors}

\begin{document}

\title[Large deviations for point processes]{Large
  deviations for point processes based on stationary sequences with heavy
  tails}


\author[H.~Hult]{Henrik Hult}
\address[H.~Hult]{Department of Mathematics, KTH, 100 44 Stockholm, Sweden}
\email{{hult@kth.se}}
\urladdr{{http://www.math.kth.se/\~{}hult}}

\author[G.~Samorodnitsky]{Gennady Samorodnitsky${}^\dagger$}
\address[G.~Samorodnitsky]{School of Operations Research and
  Industrial Engineering, Cornell University, 220 Rhodes Hall, Ithaca,
  NY 14853, USA}
\email{{gennady@orie.cornell.edu}}
\urladdr{{http://www.orie.cornell.edu/\~{}gennady/}}


\thanks{$^{\dagger}$
Hult's research was partially supported by the Swedish Research
Council. Samorodnitsky's research was partially supported by NSA grant 
H98230-06-1-0069 and ARO grant W911NF-07-1-0078
at Cornell University.}

\begin{abstract}
In this paper we propose a framework that enables the study of  large 
deviations for point processes based on stationary sequences with 
regularly varying tails. This framework allows us to keep track not of
the magnitude of the extreme values of a process, but also of the
order in which these extreme values appear. 
Particular emphasis is put on (infinite)
linear processes with random coefficients. The proposed framework 
provide a rather complete description of the joint asymptotic behavior 
of the large values of the stationary sequence. We apply the general 
result on large deviations for point processes to derive the
asymptotic decay of 
partial sum processes as well as ruin probabilities. 
\end{abstract}



\maketitle

\section{Introduction}

In some applications, including network
traffic and finance, time series are encountered where the marginal
distributions are heavy-tailed and clustering of extreme values is
observed.  More precisely, the marginal distributions have a
power-like decay and large values tend to occur at nearby points in
time, forming clusters. When studying the probability of rare events
it is usually important not only to determine the size and
frequency of clusters of extreme values but also to capture the
internal structure of the clusters. Unfortunately, in many
``standard'' limiting theorems dealing with heavy tailed processes the
fine structure of a cluster is lost in the limit, including the
ordering of the points in a cluster. This point is discussed in some
detail in Section \ref{sec:intuition} below. To overcome this problem, 
we propose a new framework for investigating large deviations for
stochastic processes with heavy tails. Specifically,
large deviations are studied at the level of point processes associated 
to the underlying stochastic process. In this way it is  possible
to preserve the fine structure of the clusters of large
values for a fairly general class of multivariate time series.

The processes studied here is the class of random coefficient linear
processes. It consists of $d$-dimensional time series $(X_k)_{k \in
  \Z}$ with the stochastic representation 
\begin{align}\label{x}
  X_k = \sum_{j \in \Z} A_{k,j}Z_{k-j}.
\end{align}
The sequence $(Z_j)_{j \in \Z}$ consists of independent and
identically distributed random vectors with values in $\R^p$.  A
generic element of this
sequence is denoted by $Z$. Each $A_{k,j}$ is a random $(d \times p)$ matrix.
It is assumed that the sequence $(\Abold_k)_{k \in \Z}$ is stationary
and each $\Abold_k$ is itself a sequence of matrices,
$\Abold_k = (A_{k,j})_{j\in \Z}$. It is assumed that the sequence
$(\Abold_k)_{k \in \Z}$ is independent of the sequence $(Z_k)_{k \in
  \Z}$. 

The probability of large values of the process $(X_k)$ depends, of
course, on the distributional assumptions on $Z$ and $A_{k,j}$. In
this paper the heavy-tailed case is considered; 
the distribution of $Z$ is assumed to be regularly varying. Certain moment
conditions will also be imposed on the random matrices
$A_{k,j}$ (see Section \ref{sec:2}).

Probability distributions with regularly varying tails have
become important building blocks in a wide variety of stochastic
models. Evidence for power-tail distributions is well documented in a
large number of applications including computer networks, telecommunications,
finance, insurance, hydrology, atmospheric sciences, geology, ecology
etc. For the multi-dimensional version of \eqref{x} the notion
of multivariate regular variation will be used.

A $d$-dimensional random vector $Z$ has
a regularly varying distribution if there exists a non-null Radon
measure $\mu$ on $\R^d \setminus \{0\}$ such that
\begin{align}
  \label{eq:rv}
  \frac{P(u^{-1} Z \in \cdot\,)}{P(|Z|>u)} \to \mu(\cdot)
\end{align}
in $\bfM_0(\R^d )$. Here $\bfM_0(\R^d)$ denotes the space of Radon
measures on $\R^d$ whose restriction to $\{|x| \geq  r\}$ is finite
for each $r > 0$, with $|\cdot|$ denoting the
Euclidean norm.  Convergence $m_n \to m$ in
$\bfM_0(\R^d)$ is defined as the  convergence $m_n(f) \to m(f)$ for
each bounded continuous function $f$ vanishing on some neighborhood of
the origin. See \cite{HL06} for more details on the space
$\bfM_0(\R^d)$. 

The limiting measure $\mu$ necessarily
obeys a homogeneity property:
there is an $\alpha > 0$ such that
$\mu(uB) = u^{-\alpha}\mu(B)$ for all Borel sets $B \subset \R^d
\setminus \{0\}$. This follows from standard regular variation
arguments (see e.g.~\cite{HL06}, Theorem 3.1). The notation
$Z \in \RV(\mu, \alpha)$ will be used for a random vector satisfying
\eqref{eq:rv}. See \cite{B00}, \cite{R87, R06}, and \cite{HL06} for
more on multivariate regular variation.

The class of stochastic models with representation \eqref{x} is quite
flexible and contains a wide range of useful time series. Here are 
some examples. 
\begin{exmp}[Linear process]\label{ex:linproc}
Let $(A_j)$ be a sequence of deterministic
real-valued $d \times p$-matrices. Then, assuming convergence, $X_k =
\sum_{j\in \Z} A_{j}Z_{k-j}$ is a linear process. It is, clearly,
stationary. The ($d$-dimensional) marginal distribution of this
process  has the representation \eqref{x}.
\end{exmp}


\begin{exmp}[SRE]  \label{ex:SRE1}
An important particular case of the random coefficient linear
process is the stationary solution of a stochastic recurrence equation
(SRE).

Assume that $p=d$, and let $(Y_k, Z_k)_{k \in \Z}$ be a sequence of
independent and identically distributed  pairs of
$d \times d$-matrices and $d$-dimensional random vectors.  Put
\begin{align*}
  \Pi_{n,m} = \left\{\begin{array}{ll} Y_n\cdots Y_m, & n \leq m,\\
  \text{Id}, & n > m,\end{array}\right.
\end{align*}
where Id is the $d \times d$ identity matrix. Under certain
assumptions assuring existence of a stationary solution of a
stochastic recurrence equation (SRE)
\begin{align} \label{e:SRE}
  X_{k} = Y_k X_{k-1} + Z_k, \quad k \in \Z\,,
\end{align}
this stationary solution can be represented by a random coefficient
linear process with $A_{k,j} = \Pi_{k-j+1,k}$, $j\geq 0$, and $A_{k,j}
= 0$, $j < 0$; \citep[e.g][]{K73}. Then the marginal distribution of
the stationary solution to the SRE is of the form \eqref{x}.
\end{exmp}

\begin{exmp}[Stochastic volatility]
  Let $(X_k)$ be the solution of the SRE in the previous example where
  we assume $X_k \in (0,\infty)^d$ a.s. Let $(V_k)$ be a sequence
  of independent and identically distributed random diagonal matrices
  independent of $(X_k)$. Then $U_k = V_k X_k$ has representation
  \begin{align*}
    U_k = \sum_{j\in \Z}  \tilde A_{k,j} Z_{k-j},
  \end{align*}
  where $\tilde A_{k,j} = V_k A_{k,j}$ and $A_{k,j}$ as in the
  previous example. The sequence $U_k$ can be intepreted as a
  stochastic volatility model where $X_k$ is the volatility.
\end{exmp}

\section{Convergence and tail behavior}
\label{sec:2}

Consider a time series $(X_k)$ with stochastic representation
\eqref{x}. Throughout this paper it is assumed that
\begin{align}\label{h}
\left.\begin{array}{ll}
    Z \in \RV(\mu,
\alpha) \text{ and}\\
\text{if $\alpha>1$, we assume additionally that $\E Z = 0$}.
\end{array} \right\}
\end{align}
To begin the study of extreme values for the time series \eqref{x} a 
first requirement is to establish conditions under which the infinite series
converge a.s.\ and determine the tail behavior of the distribution of
$X_k$. Under certain conditions results on the tail behavior
were obtained recently by
\cite{HS08I},  under a ``predictability'' assumption on the matrices
$(A_{k,j})$. Here we summarize the results and remind the reader that
in the current paper it is assumed that $(A_{k,j})$ and $(Z_j)$ are
independent.  Theorem \ref{trs} below describes the marginal tails;
for simplicity we drop the time subscript $k$ from both $X_k$ and
$A_{k,j}$.

Throughout the paper the notation $\|A\|$ is used for the operator
norm of a matrix $A$.
The summation index will be omitted when it is clear
what it is.
\begin{thm}\label{trs}
  Suppose that \eqref{h} holds and there is $0<\vep <\alpha$ such that
  \begin{align}
    \sum \E \|A_{j}\|^{\alpha-\vep}<\infty \quad \text{and} \quad \sum \E
    \|A_{j}\|^{\alpha+\vep}  &< \infty, \quad  \alpha \in (0,1)\cup(1,2),
    \label{cc1} \\
    \E\left( \sum \|
      A_{j}\|^{\alpha-\vep}\right)^{\frac{\alpha+\vep}{\alpha-\vep}}
    &<\infty, \quad \alpha \in\{1, 2\},
    \label{cc1.5} \\
    \E\left( \sum \|
      A_{j}\|^{2}\right)^{\frac{\alpha+\vep}{2}}
    &< \infty, \quad  \alpha \in
    (2,\infty).\label{cc2}
  \end{align}
  Then the series \eqref{x} converges a.s. and
  \begin{equation}\label{tailrandomsum}
    \frac{\Prob(u^{-1}X \in \cdot\,)}{\Prob(| Z|>u)} \to
    \E\Big[\sum \mu\circ A_{j}^{-1}(\cdot)\Big],
  \end{equation}
  in $\bfM_0(\R^d)$. 
\end{thm}
The right hand side of \eqref{tailrandomsum} is interpreted as
\begin{align*}
   \E\Big[\sum \mu\circ A_{j}^{-1}(B)\Big] = \E\Big[\sum \mu\{z : A_jz
   \in B\}\Big],
\end{align*}
for any Borel set $B \subset \R^d$.
When both $Z$ and $A_{k,j}$ are univariate ($d = p = 1$),
the limiting measure $\mu$ of $Z$ has the representation
\begin{align}\label{mu}
\mu(dz) = \big(w \,\alpha \, z^{-\alpha-1}I\{z > 0\} + (1-w)\,\alpha\,
(-z)^{-\alpha-1} I\{z < 0\}\big)dz
\end{align}
 for some $w \in
[0,1]$. Then \eqref{tailrandomsum} becomes
\begin{align*}
\frac{\Prob(X > ux)}{\Prob(|Z|>u)} \to  \sum \E\Big[|A_j|^\alpha
(w \ind\{A_j >0\} +   (1-w) \ind\{A_j<0\})\Big] x^{-\alpha},
\end{align*}
for each $x > 0$, with a similar expression for the negative tail.

\begin{exmp}[Linear process]
If $(X_k)$ is a linear process ($A_{k,j} = A_j$
deterministic) and $d = p = 1$, then
\begin{align*}
  \frac{\Prob(X > ux)}{\Prob(|Z|>u)} \to  \sum \Big[|A_j|^\alpha
  (w \ind\{A_j >0\} +   (1-w) \ind\{A_j<0\})\Big] x^{-\alpha}.
\end{align*}
\end{exmp}

\begin{exmp}[SRE] \label{ex:SRE}
  Suppose $(X_k)$ is the solution to the stochastic recurrence equation in
  Example \ref{ex:SRE1} with $Y$ satisfying $E\|Y\|^{\alpha+\vep} < 1$
  for some $\vep > 0$. Then, in the
  case $d = p = 1$,
  \begin{align*}
    \frac{\Prob(X > ux)}{\Prob(|Z|>u)} \to
    \frac{w(1-E(Y^+)^\alpha)+(1-w)E(Y^-)^\alpha}{(1-E(Y^+)^\alpha)^2 +
      (E(Y^-)^\alpha)^2} \,x^{-\alpha},
  \end{align*}
  see \cite{HS08I}, Example 3.3. Here, and throughout, $x^+ =
  \max\{x,0\}$ denotes the positive part of $x$, and $x^- =
  \max\{-x,0\}$ its  negative part. In particular, if $Y$ is
  nonnegative then $w = 1$,   $EY^- = 0$, and
  the expression in the last display reduces to
  \begin{align*}
    \frac{\Prob(X > ux)}{\Prob(|Z|>u)} \to (1-E Y^\alpha)^{-1} x^{-\alpha}.
  \end{align*}
\end{exmp}
\begin{exmp}[Stochastic volatility]
  Let $(X_k)$ be as in the previous example where
  $d = p = 1$ and $Y$ and $Z$ are nonnegative. Let $(V_k)$ be a sequence
  of independent and identically distributed random variables,
  independent of $(X_k)$. Suppose  $EV^{\alpha+\vep} < \infty$ for some 
  $\vep > 0$. Then $U_k = V_kX_k$ satisfies
  \begin{align*}
    \frac{\Prob(U > ux)}{\Prob(Z >u)} \to
    \frac{EV^\alpha}{1-EY^\alpha} \,x^{-\alpha},
  \end{align*}
\end{exmp}

\begin{rem} \label{rk:joint.lim}
  {\rm The following two observations will be useful for later
    reference. It follows from Remark 4.1 in \cite{HS08I} that
for any increasing truncation $n(x)\uparrow\infty$ one has
    \begin{equation} \label{e:joint.lim}
      \lim_{x\to\infty} \frac{\Prob(|\sum_{|j|>n(x)}
        A_jZ_j| > x)}{\Prob(|Z|>x)} = 0\,.
    \end{equation}
    Further, only values
    of $Z_j$ comparable to the level $x$ matter in the
    sense that
    \begin{equation} \label{e:only.large}
      \lim_{\tau\to 0}   \limsup_{x \to \infty}
      \frac{\Prob(|\sum_{j\in\bbz} A_j Z_j I\{|Z_j|\leq \tau
        x)\}| >x)}{\Prob(|Z|>x)} = 0\,.
    \end{equation}
  }
\end{rem}

\section{Why are the large deviations of  point processes needed?}
\label{sec:intuition}

In this section we discuss, somewhat informally, the  joint asymptotic
behavior of large values of the sequence $(X_k)$ in \eqref{x}. The
goal is to set up the necessary background and intuition for the
general result in Section \ref{xmain}. We consider two
special cases, that of sequences of independent and identically
distributed random variables as well as of moving average processes.

\subsection{Independent and identically distributed random variables}

Consider a sequence $(Z_k)$ of independent and identically
distributed real-valued random variables with $Z \in \RV(\alpha, \mu)$
and $\mu$ as in \eqref{mu}. As mentioned before, for $\alpha > 1$ it
is assumed that $E Z_k = 0$.  It is well known that for each $n \geq
1$ the vector 
$(Z_1, \dots, Z_n)$ is regularly
varying with limit measure $\mu^{(n)}$ concentrated on the
coordinate axes;
\begin{align*}
  \mu^{(n)}(dz_1, \dots, dz_n) = \sum_{i=1}^n \mu(dz_i)\prod_{j \neq
  i}\delta_0(dz_j) \,,
\end{align*}
where $\delta_x$ is a unit mass at $x$.
The interpretation is that, asymptotically, only one of the variables
$Z_{1}, \dots, Z_n$ will have large absolute value and each variable
is equally likely of being large.

The same intuition holds true when considering variables $Z_1, \dots,
Z_n$ in a time window  of length $n$ and letting $n \to \infty$, if the threshold increases with $n$
at an appropriate rate. Let $\gamma_n$ be a sequence with
$\gamma_n \to \infty$ and such that $n\Prob(|Z|> \gamma_n) \to 0$ as $n \to
\infty$. Then, the probability to see two different $Z$'s of size of
the order $\gamma_n$ among the variables
$Z_1,\dots, Z_n$, is small compared to seeing just one $Z$ of
size of the order $\gamma_n$. Indeed, for any $\vep > 0$,
\begin{align*}
&\frac{\Prob(\text{ there exist $1 \leq i < j \leq n$ such that
  $|Z_i| > \gamma_n \vep $ and $|Z_j| > \gamma_n \vep$
})}{\Prob(|Z_i| > \gamma_n \text{ for some } 1\leq i \leq n)} \\
&\quad \sim \frac{(n(n-1)/2)\Prob(|Z|>\gamma_n \vep)^2}{n\Prob(|Z|>
  \gamma_n)} \to 0.
\end{align*}
Here $a_n \sim b_n$ is shorthand for $\lim_{n\to \infty } a_n/b_n =
1$. 

A convenient description of the large values for the
sequence $Z_1, Z_2, \dots$ can be obtained by
considering the convergence of the point measures
\begin{align*}
  N_n = \sum_{k=1}^n \delta_{(k/n, \gamma_n^{-1} Z_k)}, \ n=1,2,\ldots,
\end{align*}
on the state space $[0,1] \times (\R^d\setminus \{0\})$. 
The assumption $n\Prob(|Z|> \gamma_n) \to 0$ as $n \to
\infty$ implies that $\gamma_n \to \infty$ too fast for a non-trivial weak
convergence of $N_n$ (described, for example, by Proposition 3.21 in
\cite{R87}). When $\gamma_n$ grows so fast, the second coordinates of
all points of the point measure $N_n$ will tend to zero with
probability 1. Since points with the zero second coordinate are
defined to be not in the state space on which the point measures live
(see, once again, \cite{R87}), it turns out that
the point measure $N_n$ converges almost
surely to the null measure, denoted $\xi_0$. Intuitively, this is
exactly the situation where large deviations in the space of point
measures might help: the
hope is to find a  sequence $r_n \to \infty$ such that
$r_n\Prob(N_n \in \cdot)$ converges to some limiting measure $m$ on the
space of point measures.

The above discussion makes it reasonable to expect
that this limiting measure, $m$,  is
concentrated on point measures with one point, corresponding, for each
$n=1,2,\ldots$,  to a
single large (at the scale $\gamma_n=n$) value of $Z_{k^*}$,
$k^*=1,\ldots, n$.  In fact, the limiting measure is
expected to be
\begin{align*}
  m(B) = (\Leb \times \mu)\{(t,z): \delta_{(t,z)} \in B\}, \ \
\text{$B$ a measurable set of measures. }
\end{align*}
The ``uniform'' coordinate $t$ is interpreted as the rescaled within
the set $\{1,\ldots, n\}$ time $k^*$ of the
large $Z_{k^*}$-value. Since all $Z_k$'s have equal probability of being
large, $t$ is ``uniformly distributed'' on $[0,1]$. The corresponding
value $z$ is governed by the limiting measure $\mu$ which describes the
large values of the $Z$-variables.
The suggested convergence can be rigorously established, as is done (in a
significantly more general setting) in Theorem \ref{mt} below.

It is possible to look at this convergence as 
the partial sum convergence of the underlying sequence $(\delta_{(k/n,
  \gamma_n^{-1} Z_k)})$ in the space of point measures. This is similar
to Sanov's theorem in the light-tailed case (see e.g.\ \cite{DZ98},
Section 6.2).

\subsection{A finite moving average}

Suppose that, in \eqref{x}, $p=d=1$, and
$A_{k,j} = A_j$ are deterministic coefficients with $A_j = 0$
if $j < 0$ or $j > q$. Then $(X_k)$ is a sequence with the representation
\begin{align*}
  X_k =  A_{0}Z_{k} + A_1 Z_{k-1} + \dots + A_{q}Z_{k-q}.
\end{align*}
Consider a time-window of length $n$ where, for now, $n$ is fixed.
That is, we consider the vector
$(X_1, \dots, X_n)$. Then, $(X_1, \dots, X_n)^T = A'(Z_{1-q}, \dots,
Z_n)^T$ where $A'$ is the $n \times (n+1+q)$-matrix
\begin{align*}
  A' = \left(\begin{array}{llllllll}
      A_q & A_{q-1}& \dots  & A_0     & 0     & \dots & \dots &0\\
      0   & A_q   & A_{q-1} &\dots    & A_0   & 0     & \dots &0\\
      \vdots   & \vdots    & \ddots & \ddots & \ddots& \ddots&  \ddots &\vdots\\
      0   &  \dots &    0   &  A_q  & A_{q-1}  & \dots  & \dots &A_0
      \end{array}\right)\,.
\end{align*}
Since the $Z$ variables are independent the vector $(Z_{1-q}, \dots,
Z_n)^T$ is regularly varying with limit measure concentrated on the
coordinate axes, just as in the previous example. That is,
asymptotically, only one variable among
$Z_{1-q}, \dots, Z_n$ will be large on the large deviations scale, and
they all have equal
probability of being large. Suppose $Z_{k^*}$ is large for some $1-q
\leq k^* \leq n$. Then, since all the other $Z_k$'s are small in
comparison to $Z_{k^*}$ we expect that $X_k$ is small for $k< k^*$ and
$k > k^*+q$ whereas for $k^* \leq k \leq k^*+q$ we have
\begin{align*}
  X_k \approx A_{k-k^*}Z_{k^*}.
\end{align*}
If we study the convergence of the sequence of measures
$\bigl(r_n\Prob(N_n \in \cdot)\bigr)$, where
\begin{align*}
  N_n = \sum_{k=1}^n \delta_{(k/n, \gamma_n^{-1} X_k)}, \,
n=1,2\ldots,
\end{align*}
is defined on the state space $[0,1] \times (\R \setminus \{0\})$,
we would expect that the limiting measure is concentrated on point
measures with $q+1$ points of the form $(t,x_i)$, with the same time
coordinate $t$ and space coordinates of the form $x_i = A_iz$ for some
$z$. In other words,  we expect the limiting measure to be
\begin{align*}
  m(B) = (\Leb \times \mu)\{(t,z): \sum_{i=0}^q
  \delta_{(t,A_i z)} \in B\}, \ \
\text{$B$ a measurable set of measures. }
\end{align*}
The clustering of extreme values is captured in the limiting measure
as there are $q+1$ points corresponding to large values of the
$X_k$'s. However, in the limit all these points have the same
time-coordinate $t$,  which means that the limiting measure does not keep
track of the order in which the large values arrived. That is, the complete
internal structure of the cluster of extreme values is not
captured. The order at which the large values arrive is, however, of
crucial importance when studying, for instance, the ruin
probabilities, or the long strange segments corresponding to the
process $(X_k)$; see e.g. \cite{asmussen:2000},  \cite{DZ98},
  \cite{MS00}, \cite{HLMS05}. Therefore, information is lost in the limit.

Our suggestion for resolving this loss of information is via
considering point measures similar to the measures $N_n$ above, but
enlarging the dimension of the state space so that each point of the
point measure describes more than one value of the process
$(X_k)$. It is intuitive that for a finite moving average of the above
example it is enough to keep track of $q+1$ consecutive observations
of the stationary process, and this tells us how large the state space
of the point measures should be. Specifically, we will consider the
point measures
\begin{align*}
  \sum_{k=1}^n \delta_{(k/n, \gamma_n^{-1} (X_k, X_{k-1}, \dots,
X_{k-q}))}, \, n=1,2\ldots\,.
\end{align*}
The above discussion should make it intuitive that, for such point
processes, the limiting measure  in a large deviations procedure
should be concentrated on point measures with $2q+1$ points of the form
\begin{align*}
(t, (A_0z, 0, \dots,0)), (t,
(A_1z,A_0z,0,\dots,0)), \dots, (t, (0,\dots,0,A_qz)).
\end{align*}
Notice that the information about the order in which the extreme
values arrived can be
obtained because the space coordinates are simply shifts of
each other.

In general, the complete information
on the extreme values of the process will only be completely preserved
if one keeps track of infinite (or increasing with $n$) number of
observations of the process $(X_k)$. This is possible to do, but we
have chosen not pursue this last possibility because it complicates
significantly the technical details of the construction of the point
measures and working with these measures. Instead, we have chosen to
to construct point measures based on finitely many consecutive observations
of the stationary process, as if it were a finite moving average. In
applications we are considering, this turns out to be sufficient via
an application of a truncation argument.

\section{Large deviations for point processes: the main result}
\label{xmain}

We start with specifying the precise assumptions on the normalizing
sequence $(\gamma_n)_{n\geq 1}$ that are needed to obtain a large
deviation scaling. We assume that
\begin{equation}\label{gamman}
  \left.\begin{array}{rl}
      (Z_1+\dots+Z_n)/\gamma_n \to 0, & \text{in probability and}\\
      \gamma_n/\sqrt{n^{1+\vep}} \to \infty, & \text{for some $\vep>0$
        if $\alpha = 2$},\\
      \gamma_n /\sqrt{n\log n} \to \infty, & \text{if $\alpha > 2$.}
    \end{array}\right\}
\end{equation}
Note that these conditions are exactly the same as those that were
used in  Theorem 2.1 in \citep{HLMS05} to obtain a functional level
large deviation result for the partial sums of independent and identically
distributed random vectors. If we set
\begin{align*}
  r_n = \frac{1}{n\Prob(|Z| > \gamma_n)}\,,
\end{align*}
then $r_n \to \infty$ as $n\to \infty$ and it turns out that
normalizing the probability measures of the point processes by $(r_n)$ is
the correct normalization to obtain a large deviation result.

For $q\geq 0$ define a point measure $N_n^q$
on the space $\Eb^q = [0,1] \times (\R^{d(q+1)} \setminus \{0\})$ by
\begin{align}\label{Nn}
  N_n^q = \sum_{k=1}^n
  \delta_{(k/n, \gamma_n^{-1} X_{k}, \gamma_n^{-1} X_{k-1},
    \dots,\gamma_n^{-1} X_{k-q})}.
\end{align}
We will show that the sequence of measures on the space of point
measures,
\begin{align*}
  m_n^q(\cdot) = r_n\Prob(N_n^q \in \cdot\,), \ n\geq 1\,,
\end{align*}
converges in the appropriate sense and compute the limiting
measure, called $m^q$, for any $q \geq 0$. The limiting measure will
give us a partial description of the extremal behavior of the
sequence $(X_k)$. This description will become more and more detailed
as the number $q$ is taken larger and larger.

A technical framework suitable for studying this problem is provided in
the Appendix, and we are using the notation introduced there. Let 
$\Np^q = \Np(\Eb^q)$ be the space of point
measures on $\Eb^q$ equipped with the vague topology.
The convergence $m_n^q \to m^q$ takes place in the space
$\bfM_0(\Np^q)$, the space of Radon measures on $\Np^q$ that are finite on
sets of the form $\{\xi: d(\xi,\xi_0) > r\}$, for each $r > 0$ (see
the Appendix). Here
$\xi_0$ denotes the null measure and $d(\cdot, \cdot)$ the metric  on
$\Np^q$ given by \eqref{metd}. With this metric, $(\Np^q, d)$ is a
complete separable metric space.

For a sequence of $d \times p$-matrices $\Abold =
(A_{k,j})_{j,k\in\Z}$ and $(t,z)\in [0,1]\times \R^p\setminus \{0\})$
we write
\begin{align*}
  T_{\Abold,q}(t,z) = \sum_{j \in \Z} \delta_{(t, A_{j,j} z,
A_{j-1,j-1}z,\dots,A_{j-q,j-q}z)}\,.
\end{align*}
Under certain conditions on the matrices in $A_{k,j}$,  $T_{\Abold,q}$
will be a map from $[0,1] \times\R^p\setminus \{0\})$ into the space
$\Np^q$. 

We are now ready to state the main result of this paper.

\begin{thm}\label{mt}
Suppose that \eqref{h}, \eqref{cc1}--\eqref{cc2}, and
\eqref{gamman} hold. Then, for any $q \geq 0$,
the stationary process $(X_k)_{k\in\Z}$ in \eqref{x} satisfies
 \begin{equation}\label{limitmeasure}
    m_n^q(\cdot) = r_n\Prob(N_n^q \in \cdot\,) \to \E[(\Leb \times
    \mu)\circ T_{\Abold,q}^{-1}(\cdot)] =: m^q(\cdot)
  \end{equation}
in $\bfM_0(\Np^q)$. In particular, $T_{\Abold,q}$ is, with
probability 1, a
map from $[0,1] \times \R^p\setminus \{0\})$ into the space $\Np^q$.
\end{thm}

\begin{rem} \label{rk:cts.on.sphere}
For any $a>0$ the measure $m^q$ on  $\Np^q$ defined in
\eqref{limitmeasure} satisfies
\begin{align*}
&m^q\Big\{\xi:\, \xi \big([0,1] \times \big\{(x_0,\dots,x_q):\,
|x_i|=a, \, \text{ some }  i\in\{0,\dots,q\}\big\} \big)>0\Big\}
\\ &
= E\Big[\mu\big\{z : \sum_{j \in \Z}
\delta_{(A_{j,j}z,\dots,A_{j-q,j-q}z)}((x_0,\dots,x_q)\!:\!
|x_i|\! =\! a,\!
\text{ some }\! i \in\{0,\dots, q\}) \!> \!0\big\}\Big]\\
& \leq \sum_{j \in \Z} \E\mu\bigl\{z:\, |A_{j,j}z|=a \bigr\}=0
\end{align*}
by the scaling property of the measure $\mu$. This fact is useful for
establishing continuity almost everywhere with respect to the measure
$m^q$ of various mappings.
\end{rem}

\begin{exmp}[Independent and identically distributed random vectors]
  For a sequence of independent and identically distributed random
  vectors we have $A_{k,j} = A\ind\{j=0\}$, where  $A$ is a fixed
  $d \times p$-matrix, and, hence, for $q = 0$, the limiting
  measure $m^0$  is given by $m^0(\cdot) = (\Leb \times \mu)\circ
  T_{iid}^{-1}(\cdot)$ where $T_{iid}$ is the mapping
  \begin{align*}
    T_{iid}(t,z) = \delta_{(t, z)}.
  \end{align*}
\end{exmp}
\begin{exmp}[Linear process] \label{ex:linproc2}
  For a linear process the
  matrices $A_{k,j}= A_j$, $j\in \Z$ are deterministic. The limiting
  measure $m^q$ is given by $m^q(\cdot) = (\Leb \times \mu)\circ
  T_{A,q}^{-1}(\cdot)$,
  with  the mapping $T_{A,q}$ simplifying to
  \begin{align*}
    T_{A,q}(t,z) = \sum_{j \in \Z} \delta_{(t,A_j z, A_{j-1}z,\dots,A_{j-q}z)}.
  \end{align*}
\end{exmp}

\subsection*{Proof of Theorem \ref{mt}}
By Theorem \ref{sc} we need to prove that the measure $m^q$ in
\eqref{limitmeasure} belongs to $\bfM_0(\Np^q)$, and that
\begin{equation} \label{e:prove}
  m_n(F_{g_1,g_2,\vep_1,\vep_2}) \to  m^q(F_{g_1,g_2,\vep_1,\vep_2})
\end{equation}
for all  Lipschitz functions $g_1,g_2 \in C_K^+(\Eb^q)$
and $\vep_1,\vep_2 > 0$, where the functions
$F_{g_1,g_2,\vep_1,\vep_2}$ are given in \eqref{e:new.Laplace} in the
Appendix. For the first statement, it is enough to prove that for each
$\delta>0$,
$$
\E S_\delta=:\ E\left( \sum_{j \in \Z} \ind\{
  \|A_j\|>\delta\}\right)<\infty\,.
$$
This is an easy consequence of conditions
\eqref{cc1}--\eqref{cc2}. For example, if $0<\alpha\leq 2$, then for
$0<\vep<\alpha$,
$$
\E S_\delta\leq \delta^{-(\alpha-\vep)}  \sum_{j \in \Z} \E
\|A_{j}\|^{\alpha-\vep}<\infty \,,
$$
and the case $\alpha>2$ is similar.

We now prove \eqref{e:prove}. Note that
\begin{align}
  &m_n(F_{g_1,g_2,\vep_1,\vep_2}) = \label{e:first.rhs} \\
  &r_n \E\left[\left(1-\exp\Big\{-\Big[\sum_{k=1}^n
      g_{1}(\frac{k}{n},
      \frac{X_{k}}{\gamma_n},\dots,\frac{X_{k-q}}{\gamma_n}) -
      \vep_1\Big]_+\Big\}\right)\right.\nonumber\\
  &\qquad \times \left. \left(1-\exp\Big\{-\Big[\sum_{k=1}^n
      g_{2}(\frac{k}{n},
      \frac{X_{k}}{\gamma_n},\dots,\frac{X_{k-q}}{\gamma_n}) -
      \vep_2\Big]_+\Big\}\right)\right].\nonumber
\end{align}
The first step is to truncate the infinite sum in the definition of $X_{k}$,
replacing $X_{k}$ by $\sum_{|j|\leq
  J_n}A_{k,j}Z_{k-j}$, as follows.  Let $(J_n)$ be a sequence of
positive numbers such that $J_n \to \infty$ and
\begin{equation} \label{e:Jn}
\left. \begin{array}{rl}
J_n = o(n) & \text{if $0<\alpha<1$,} \\
J_n = o\bigl( \min( n,\gamma_n/l(\gamma_n))\bigr) & \text{if $\alpha=1$,} \\
J_n = o\bigl( \min( n,\gamma_n)\bigr) & \text{if $\alpha>1$,}
\end{array}\right\}
\end{equation}
where for $x>0$, $l(x) = E\bigl(|Z|I\{|Z|\leq x\}\bigr)$. The
conditions on the asymptotic growth of $J_n$ will be used below.

By Lemma   \ref{lem1}
there is a sequence $\beta_n\downarrow 0$ such that
$$
r_n\Prob\Big(\max_{\scriptsize{1\leq k \leq
    n}} \frac{1}{\gamma_n}\Big|\sum_{|j|>J_n}
A_{k,j}Z_{k-j}\Big|> \beta_n\Big) \to 0
$$
as $n\to\infty$. Therefore, the expression in the right hand side of
\eqref{e:first.rhs} is within $o(1)$ of
\begin{align*}
r_n \!\E\Big[\Big(1\!-\!\exp\Big\{-\Big[\sum_{k=1}^n
    g_{1}\Big(\frac{k}{n}, R_{k,n}+&\frac{1}{\gamma_n}\sum_{|j|\leq J_n} A_{k, j}
    Z_{k-j}, \dots, \\
    R_{k-q,n}+&\frac{1}{\gamma_n}\sum_{|j|\leq J_n} A_{k-q, j}
    Z_{k-q-j}\Big)-\vep_1\Big]_+\Big\}\Big) \\
   \times \Big(1-\exp\Big\{-\Big[\sum_{k=1}^n
    g_{2}\Big(\frac{k}{n}, R_{k,n}+&\frac{1}{\gamma_n}\sum_{|j|\leq J_n} A_{k, j}
    Z_{k-j},\dots, \\
    R_{k-q,n}+&\frac{1}{\gamma_n}\sum_{|j|\leq J_n} A_{k-q, j}
    Z_{k-q-j}\Big)-\vep_2\Big]_+\Big\}\Big)\Big] \\
  :=r_n\E(\Theta_n)\,, \qquad \qquad \qquad \qquad \quad &
\end{align*}
where $(R_{k,n})$ are random variables satisfying $|R_{k,n}|\leq
\beta_n$ for all $k,n$.

  To proceed we use the intuitive idea that only one of the $Z$'s
  is likely to be large. Take
  $\tau>0$. The above expression can be decomposed as
  \begin{align}
    &r_n\E(\Theta_n\ind{\{\textrm{ all $|Z_{-J_n-q+1}|, \dots,
    |Z_{n+J_n}|$ are less than $\tau \gamma_n$}\}}) \nonumber \\
      & \quad +
    r_n\E(\Theta_n \ind{\{\textrm{exactly one of $|Z_{-J_n-q+1}|, \dots,
    |Z_{n+J_n}|$ exceeds $\tau \gamma_n$}\}}) \nonumber \\
    &\quad +
    r_n\E(\Theta_n \ind{\{\textrm{at least two of $|Z_{-J_n-q+1}|, \dots,
    |Z_{n+J_n}|$ exceed $\tau \gamma_n$}\}}) \nonumber \\
    &= r_n\E\Bigl(\Theta_n \ind{\Bigl\{\bigcap_{t=-J_n-q+1}^{n+J_n} |Z_t|\leq
    \tau \gamma_n\Bigr\}}\Bigr) \label{t1} \\
    &\quad + r_n\E\Bigl(\Theta_n \ind{\Bigl\{\bigcup_{t=-J_n-q+1}^{n+J_n}
    \bigcap_{{s=-J_n-q+1, \ldots, n+J_n}\atop{s\neq t}} \{|Z_{t}|>
      \tau \gamma_n, |Z_s| \leq \tau \gamma_n\}\Bigr\}}\Bigr)
    \label{t2}\\
    &\quad + r_n\E\Bigl(\Theta_n \ind{\Bigl\{\bigcup_{t=-J_n-q+1}^{n+J_n}
    \bigcup_{{s=-J_n-q+1, \ldots, n+J_n}\atop{s\neq t}}\{|Z_t|> \tau \gamma_n,
    |Z_s|> \tau \gamma_n \}\Bigr\}}\Bigr). \label{t3}
  \end{align}
  We claim that the main contribution comes from \eqref{t2} and that
  the contributions from the other terms vanish as $n\to\infty$ and
  then $\tau \to 0$. Let us start with
  \eqref{t1}. Recall that  $g_{1}$  and $g_{2}$ have compact supports in
  $\Eb^q = [0,1]\times (\R^{d(q+1)}\setminus\{0\})$. Hence, there is a
   $\delta >0$ such that $\bigl([0,1] \times
  \{(x_0,\dots,x_q): \max\{|x_0|,\dots,|x_q|\} < \delta\}   \bigr)\cap
  \{\textrm{support}(g_{1})\cup
  \textrm{support}(g_{2})\} =    \emptyset$. On the set
  $\cap_{t=-J_n-q+1}^{n+J_n} \{|Z_t|\leq   \tau \gamma_n\}$ we have, for
  large $n$,
  \begin{align*}
    &r_n\E\Bigl(\Theta_n
    \ind{\Bigl\{\cap_{t=-J_n-q+1}^{n+J_n} |Z_t|\leq \tau
    \gamma_n\Bigr\}}\Bigr) \\
    &\leq
    r_n\E\Big(\Theta_n I{\Big\{\bigcup_{k=1-q}^n \Big\{ \big|R_{k,n} +
    \frac{1}{\gamma_n} \sum_{|j|\leq J_n} A_{k, j}
    Z_{k-j}I\{|Z_{k-j}|\leq \tau \gamma_n\}\big| > \delta \Big\}}\Big) \\
    & \leq r_n \Prob\Big(\bigcup_{k=1-q}^n \Big\{\big|R_{k,n}+
    \frac{1}{\gamma_n}\sum_{|j|\leq J_n} A_{k, j}
    Z_{k-j}I\{|Z_{k-j}|\leq \tau \gamma_n\}\big| >\delta \Big\}\Big)\\
    & \leq r_n (n+q)\,\Prob\Big(\big|\sum_{|j|\leq J_n}
    A_{0, j}\,
    Z_{j}I\{|Z_{j}|\leq \tau \gamma_n\}\big| > \gamma_n \delta/2 \Big) \to 0,
  \end{align*}
  as $n\to \infty$ and then $\tau \to 0$, by appealing to
  \eqref{e:only.large} (the last inequality used the fact that
  $\beta_n \downarrow 0$).
  For \eqref{t3} we observe that for any $\tau>0$
  \begin{align*}
    &r_n\E\Bigl(\Theta_n \ind{\Bigl\{\bigcup_{t=-J_n-q+1}^{n+J_n}
\bigcup_{{s=-J_n-q+1, \ldots, n+J_n}\atop{s\neq t}}\{|Z_t|> \tau \gamma_n,
      |Z_s|> \tau \gamma_n \}\Bigr\}}\Bigr) \\
    &\quad \leq r_n\Prob\Bigl(\bigcup_{t=-J_n-q+1}^{n+J_n}
\bigcup_{{s=-J_n-q+1, \ldots, n+J_n}\atop{s\neq t}}\{|Z_t|> \tau \gamma_n,
      |Z_s|> \tau \gamma_n \}\Bigr)\\
    &\quad \leq r_n(n+q+2J_n)^2 \Prob(|Z|>\tau \gamma_n )^2 \to 0
  \end{align*}
  as $n\to\infty$ by the definition of $r_n$
  and the fact that $J_n/n$ is bounded. Hence, as claimed, the main
  contribution comes from     \eqref{t2}. Since the union is disjoint we
  may rewrite \eqref{t2} as
  \begin{align}
    \sum_{t=-J_n-q+1}^{n+J_n} \! r_n\! \E\!\Big[\Big(1\! - \!
    \exp\Big\{-\Big[\sum_{k=1}^n
    g_{1}\Big(\frac{k}{n}, R_{k,n}+&
    \frac{1}{\gamma_n}\sum_{|j|\leq J_n} \! A_{k, j}
    Z_{k-j},\dots, \nonumber \\
    R_{k-q,n}+&
    \frac{1}{\gamma_n}\sum_{|j|\leq J_n} \! A_{k-q, j}
    Z_{k-q-j}\Big)-\vep_1\Big]_+\Big\}\Big) \nonumber \\
    \times \Big(1\!-\!\exp\Big\{-\Big[\sum_{k=1}^n
    g_{2}\Big(\frac{k}{n}, R_{k,n}+&
    \frac{1}{\gamma_n}\sum_{|j|\leq J_n} \! A_{k, j}
    Z_{k-j},\dots, \nonumber \\
    R_{k-q,n}+&
    \frac{1}{\gamma_n}\sum_{|j|\leq J_n} \! A_{k-q, j}
    Z_{k-q-j}\Big)-\vep_2\Big]_+\Big\}\Big)
    \nonumber \\
    \times \ind\{|Z_t|> \tau \gamma_n, \,
  |Z_s|\leq \tau \gamma_n,  \text{all } s=-J_n&-q+1, \ldots,
    n+J_n, \, s\neq t\} \Big].  \label{p1}
\end{align}
As $|Z_t|$ is large and $|Z_s|$ is small, $s \neq t$, we can practically
ignore the
contribution from the $|Z_s|$-terms. To be precise we claim that the above
expression is
asymptotically equal (written $a_n \sim b_n$) to
\begin{align}
  \sum_{t=-J_n-q+1}^{n+J_n} \!\!\!\!\!\!\! r_n \E\!\Big[\!&\Big(1\! - \!
  \exp\! \Big\{\!\!-\!\!\Big[\sum_{k=1}^n
  g_{1}\big(\frac{k}{n},
  R_{k,n}\!+\! \frac{1}{\gamma_n}\!\! \sum_{|j|\leq J_n} \!\!
  A_{k, j}
  Z_{k-j}\ind\{t\!=\!k\!-\!j\},\dots, \nonumber \\
  & \qquad R_{k-q,n}+
    \frac{1}{\gamma_n}\sum_{|j|\leq J_n} \! A_{k-q, j}
    Z_{k-q-j}\ind\{t\!=\!k\!-q-\!j\}\big)\!-\!\vep_1\Big]_+\Big\}\Big)
  \nonumber \\
  \times & \Big(1\! - \!
  \exp\! \Big\{\!\!-\!\!\Big[\sum_{k=1}^n
  g_{2}\big(\frac{k}{n},
  R_{k,n}\!+\!\frac{1}{\gamma_n}\!\! \sum_{|j|\leq J_n} \!\!
  A_{k, j} Z_{k-j}\ind\{t\!=\!k\!-\!j\}, \dots, \nonumber \\
  & \qquad
  R_{k-q,n}\!+\!\frac{1}{\gamma_n}\!\! \sum_{|j|\leq J_n} \!\!
  A_{k-q, j} Z_{k-q-j}\ind\{t\!=\!k\!-q-\!j\}\big)\!-\!\vep_1\Big]_+\Big\}\Big)
  \nonumber \\
  \times & \ind\{|Z_t|> \tau \gamma_n\}\Big]
  \nonumber \\
  &\hspace{-1cm}=: \sum_{t=-J_n-q+1}^{n+J_n}r_n\E[\Theta'_n \ind\{|Z_t|> \tau
  \gamma_n\} ]\,.\label{p2}
\end{align}
For now we postpone the proof that \eqref{p1} $\sim$ \eqref{p2} and
proceed, instead, with analyzing \eqref{p2}. We
can rewrite \eqref{p2} as
\begin{align*}
  \sum_{t=-J_n-q+1}^{n+J_n} \!\!\!\!\! r_n \E\!\Big[& \Big(1\! - \!
  \exp\Big\{\!\!-\!\!\Big[\sum_{k=1}^{n} g_{1}\big(\frac{k}{n}, \!
    R_{k,n}\!+ \! \frac{1}{\gamma_n}  A_{k,k-t}
    Z_{t}\ind\{|k\!- \! t|\!\leq \!J_n\},\dots, \\ \qquad & R_{k-q,n}\!+ \!
    \frac{1}{\gamma_n}  A_{k-q,k-q-t}
    Z_{t}\ind\{|k\!-q- \! t|\!\leq \!J_n\}\big)\!-\!\vep_1\Big]_+\Big\}\Big)\\
  \times &\Big(1\!-\!\exp\Big\{\!\!-\!\!\Big[\sum_{k=1}^n
    g_{2}\big(\frac{k}{n}, \! R_{k,n} \! + \!
    \frac{1}{\gamma_n}  A_{k, k-t}
    Z_{t}\ind\{|k\! - \! t|\!\leq\! J_n\},\dots, \\ & \qquad R_{k-q,n}\!+ \!
    \frac{1}{\gamma_n}  A_{k-q,k-q-t}
    Z_{t}\ind\{|k\!-q- \! t|\!\leq \!J_n\}\big)\!-\!\vep_2\Big]_+\Big\}\Big)\\
  \times & \ind\{|Z_t|> \tau \gamma_n\}\Big]
\end{align*}
In the sequel, as the subscripts change, we will write $R_n$ instead
of the proper $R_{k,n}$ corresponding to the appropriate
subscripts. We will not impose any assumtions on these random
variables apart from the fact that $|R_n|\leq \beta_n$ for all
$n$. With $l=k-t$ we can rewrite the above expression as
\begin{align*}
  \sum_{t=-J_n-q+1}^{n+J_n}\!\! r_n \E\!\Big[&\Big(1\! - \!
  \exp\Big\{\!\!-\!\!\Big[\sum_{l=1-t}^{n-t} \!
  g_{1}\Big(\frac{t+l}{n}, \, R_n\!+\!
  \frac{1}{\gamma_n} A_{t+l,l}Z_{t}I\{|l|\! \leq\!
  J_n\}, \dots,  \\ & \qquad R_n\!+\!
  \frac{1}{\gamma_n} A_{t+l-q,l-q}Z_{t}I\{|l-q|\! \leq\!
  J_n\}\Big)\!-\!\vep_1\Big]_+\Big\}\Big)\\
  \times& \Big(1\!-\!\exp\Big\{\!\!-\!\!\Big[\sum_{l=1-t}^{n-t} \!
  g_{2}\Big(\frac{t+l}{n}, \, R_n\!+\!
  \frac{1}{\gamma_n}A_{t+l, l}Z_{t}1\{|l| \!\leq\!
  J_n, \dots, \\ & \qquad  R_n\!+\!
  \frac{1}{\gamma_n} A_{t+l-q,l-q}Z_{t}I\{|l-q|\! \leq\!
  J_n\}\}\Big)\!-\!\vep_2\Big]_+\Big\}\Big)\\
  \times & \ind\{|Z_t|> \tau \gamma_n\}\Big] \,.
\end{align*}
By stationarity we may replace $A_{t+l-i,l}$ by $A_{l-i,l}$, $i = 0,\dots,q$,
and conditioning on $Z_t$ the above equals
\begin{align*}
  \int_{|z|>\tau}\,\!\!\sum_{t=-J_n-q+1}^{n+J_n}\!\!\!\! r_n
  \E\!\Big[&\Big(1\! - \!
  \exp\Big\{\!\!-\!\!\Big[\sum_{l=1-t}^{n-t} \!
  g_{1}\Big(\frac{t+l}{n}, \, R_n\!+\!
  A_{l,l} z I\{|l| \!\leq\! J_n\},  \dots, \\ & \qquad R_n\!+\!
  A_{l-q,l-q} z I\{|l-q| \!\leq\! J_n\} \Big)\!-\!\vep_1\Big]_+\Big\}\Big)\\
  \times & \Big(1\!- \!
  \exp\Big\{\!\!-\!\!\Big[\sum_{l=1-t}^{n-t}\!
  g_{2}\Big(\frac{t+l}{n}, \, R_n\!+\!
  A_{l, l} z I\{|l| \!\leq\! J_n\},\dots, \\ & \qquad R_n\!+\!
  A_{l-q,l-q} z I\{|l-q| \!\leq\! J_n\}\Big)\!-\!\vep_2\Big]_+\Big\}\Big)\Big]\\
  \times & \Prob(\gamma_n^{-1}Z_t \in dz)\\
  & \hspace{-2cm}=: \int_{|z|>\tau} \kappa_n(z)\, r_n n
  \Prob(\gamma_n^{-1}Z \in dz) \\
  &\hspace{-2cm} =:  \int_{|z|>\tau} \kappa_n(z)\, \mu_n(dz) \sim
  \int_{|z|>\tau} \tilde \kappa_n(z)\, \mu_n(dz)\,,
\end{align*}
where
\begin{align*}
  &\tilde \kappa_n(z) = \sum_{t=-J_n+1}^{n+J_n}\!\!\!\!\E\!\Big[\Big(1\! - \!
  \exp\Big\{\!-\Big[\sum_{l=1-t}^{n-t} \!
  g_{1}\Big(\frac{t+l}{n}, A_{l,l}zI\{|l|\leq
  J_n\},\dots, \\ & \qquad
  A_{l-q,l-q} z I\{|l-q| \!\leq\! J_n\}\}\Big)-\vep_1\Big]_+\Big\}\Big)\\
  &\qquad \times \Big(1-\exp\Big\{-\Big[\sum_{l=1-t}^{n-t}
  g_{2}\Big(\frac{t+l}{n}, A_{l, l}
  z1\{|l|\leq J_n\},\dots, \\ & \qquad
  A_{l-q,l-q} z I\{|l-q| \!\leq\!
  J_n\}\}\Big)-\vep_2\Big]_+\Big\}\Big)\Big]\frac{1}{n}\,,
\end{align*}
and we have used the uniform continuity of
the functions $g_1$ and $g_2$ and the fact that $|R_n| \leq \beta_n
\downarrow 0$. We claim
that, as $n\to\infty$,
\begin{equation} \label{e:conv.kapp}
\int_{|z|>\tau} \tilde \kappa_n(z)\, \mu_n(dz) \to
\int_{|z|>\tau}  \kappa(z)\, \mu(dz)\,,
\end{equation}
where
    \begin{align*}
      \kappa(z) & = \int_0^1 \E\!\Big[\Big(1\! - \!
      \exp\Big\{\!-\Big[\sum_{l\in\Z} \!
      g_{1}\Big(t, A_{l,l}
      z,\dots, A_{l-q,l-q}z\Big)-\vep_1\Big]_+\Big\}\Big)\\
      &\quad \times \Big(1-\exp\Big\{-\Big[\sum_{l \in \Z}
      g_{2}\Big(t, A_{l, l}
      z,\dots, A_{l-q,l-q}z\Big)-\vep_2\Big]_+\Big\}\Big)\Big]dt.
    \end{align*}
Note, first of all, that $\mu_n \to \mu$ in $\mathbf{M}_0(\R^d)$. Since the
functions $(\tilde \kappa_n)$ and $\kappa$ are uniformly bounded, it
is enough to prove the convergence in \eqref{e:conv.kapp} when
integrating over the set $\{\tau<|z|<M\}$ for any finite
$M>\tau$. Using the fact $J_n/n\to 0$ one needs to check
that for any $K$
\begin{equation} \label{e:conv.kapp1}
\int_{\tau<|z|<M}  \kappa_n^{(K)}(z)\, \mu_n(dz) \to
\int_{\tau<|z|<M}  \kappa(z)\, \mu(dz)\,,
\end{equation}
with
    \begin{align*}
      \kappa_n^{(K)}(z) = \sum_{t=-K+1}^{n+K}\!\!\!\!\E\!\Big[\Big(1\! - \!
      \exp\Big\{\!-\Big[\sum_{l=1-t}^{n-t} \!
      g_{1}(& \frac{t+l}{n}, A_{l,l} z I\{|l|\leq
      J_n\},\dots, \\
      &A_{l-q,l-q} z I\{|l-q|\leq
      J_n\})-\vep_1\Big]_+\Big\}\Big)\\
      \times \Big(1-\exp\Big\{-\Big[\sum_{l=1-t}^{n-t}
      g_{2}(& \frac{t+l}{n}, A_{l, l} zI\{|l|\leq
      J_n\},\dots, \\
      &A_{l-q,l-q} z I\{|l-q|\leq
      J_n\})-\vep_2\Big]_+\Big\}\Big)\Big]\frac{1}{n}\,.
     \end{align*}
Recall that the supports of $g_1$ and $g_2$ does not intersect the set
$[0,1]\times \{(x_0,\dots,x_q): \max\{|x_0|,\dots,|x_q|\} < \delta\}$,
some $\delta > 0$.  The assumptions \eqref{cc1}-
\eqref{cc2} imply that
\begin{align}\label{p5}
\Prob(\|A_{l, l}\| \geq \delta/M \ \ \text{for some $|l|\geq K$})
\to 0 \ \ \text{as $K\to\infty$.}
\end{align}
Since the limit in \eqref{e:conv.kapp1} does not depend on $K$, one
may replace $\kappa_n^{(K)}$ in it (but still using the same notation) with
    \begin{align*}
      \kappa_n^{(K)}(z) = \sum_{t=-K+1}^{n+K}\!\!\!\!\E\!\Big[& \Big(1\! - \!
      \exp\Big\{\!-\Big[\sum_{l\in\Z} \!
      g_{1}(\frac{t+l}{n}, A_{l,l} z,\dots,A_{l-q,l-q}
      z)-\vep_1\Big]_+\Big\}\Big)\\
      \times & \Big(1-\exp\Big\{-\Big[\sum_{l\in\Z}
      g_{2}(\frac{t+l}{n}, A_{l, l}
      z,\dots,A_{l-q,l-q} z)-\vep_2\Big]_+\Big\}\Big)\Big]\frac{1}{n}\,.
     \end{align*}
However, $\kappa_n^{(K)}\to \kappa$ uniformly (in $z$). Therefore,
\eqref{e:conv.kapp1} follows, e.g.  by
\citet[][Theorem 5.5]{B68}.  Having now established
\eqref{e:conv.kapp}, we let $\tau \to 0$ to obtain
    \begin{align*}
      \int_{|z|>\tau} \kappa(z) \mu(dz) \to \E[(\Leb
      \times \mu)\circ
      T_{\Abold}^{-1}(F_{g_1,g_2,\vep_1,\vep_2})],
    \end{align*}
    as required.

    It remains only to prove the asymptotic equivalence \eqref{p1}
    $\sim$ \eqref{p2}. Denote
    $C_n = \{ -J_n-q+1, \ldots, n+J_n\}$. Substracting \eqref{p2} from
    \eqref{p1} yields
    \begin{align}
     & \sum_{t=-J_n-q+1}^{n+J_n} r_n\Big(\E[\Theta_n
      \ind\{|Z_t|>\tau \gamma_n, |Z_s|\leq \tau\gamma_n, \ \text{all
      $s\in C_n, \, s\neq t$}\}] \nonumber \\
      &\hskip 1in- \E[\Theta'_n
      \ind\{|Z_t|>\tau \gamma_n\}]\Big)\nonumber \\
      &= \sum_{t=-J_n-q+1}^{n+J_n} r_n\E[(\Theta_n-\Theta'_n)
      \ind\{|Z_t|>\tau \gamma_n, |Z_s|\leq \tau\gamma_n\ \text{all
      $s\in C_n, \, s\neq t$}\}] \label{q1} \\
      & +
       \sum_{t=-J_n-q+1}^{n+J_n} r_n \E[\Theta'_n \ind\{|Z_t|>\tau
       \gamma_n\}(1-\ind\{|Z_s|\leq \tau \gamma_n, \ \text{all
      $s\in C_n, \, s\neq t$}\})]\Big). \label{q2}
    \end{align}
    Since $\Theta'_n \leq 1$,  we can bound \eqref{q2} by
    \begin{align*}
      \sum_{t=-J_n-q+1}^{n+J_n} r_n \Prob(|Z|>\tau \gamma_n) \left(
        1-\Bigl( 1-\Prob(|Z|>\tau \gamma_n)\Bigr)^{n+q+2J_n}\right)\to 0,
    \end{align*}
    as $n\to\infty$ by the choice of $r_n$ and the fact that
    $J_n/n\to 0$.  To handle \eqref{q1} we use Lemma
    \ref{l:major.bound}. This completes the proof.
\hfill{$\square$}

\begin{lem}\label{lem1}
  For the stationary process $(X_k)_{k\in\Z}$ in \eqref{x} we have,
  under the   assumptions \eqref{cc1} - \eqref{cc2} and \eqref{e:Jn},
  \begin{align*}
     \lim_{n\to \infty} r_n\Prob\Big(\max_{1\leq k \leq n} \Big|\sum_{|j|>J_n}
     A_{k,j}Z_{k-j}\Big|> \gamma_n \vep\Big) =
     0
   \end{align*}
for any $\vep>0$.
\end{lem}
\begin{proof}
By stationarity we have
   \begin{align*}
     &r_n\Prob\Big(\max_{1\leq k \leq n} \Big|\sum_{|j|>J_n}
     A_{k,j}Z_{k-j}\Big|> \gamma_n \vep\Big)
     \quad \leq r_n n \Prob\Big(\Big|\sum_{|j|>J_n}
     A_{k,j}Z_{k-j}\Big|> \gamma_n \vep\Big).
   \end{align*}
   Using Remark \ref{rk:joint.lim} and the definition of $r_n$,
   we see that the above    expression is   bounded from above by
   $$
   o(1)\, n\, r_n \Prob(|Z|>\gamma_n \vep) \to 0
 $$
 as $n\to\infty$.
\end{proof}

\begin{lem}\label{l:major.bound}
Let $\tilde \Delta_n$ be the sum in \eqref{q1}. Then
$\lim_{\tau \to 0}\limsup_{n\to\infty}\tilde\Delta_n =  0$.
\end{lem}
\begin{proof}
Note that by taking norms it is enough to consider the one dimensional case
$d=p=1$. Furthermore, it is clearly enough to consider a single
function $g$ and $\vep>0$ and prove that
\begin{equation} \label{Delta.alt}
\lim_{\tau\to 0}\limsup_{n\to\infty}r_n\, \Delta_n = 0 \,,
\end{equation}
where
\begin{align*}
  \Delta_n = \sum_{t\in C_n} E\Biggl[ \exp\Big\{& -\Big[\sum_{k=1}^n
  g\Big(\frac{k}{n}, R_{k,n}+\frac{1}{\gamma_n}\sum_{|j|\leq J_n} A_{k, j}
  Z_{k-j}, \dots, \\ & 
  R_{k-q,n}+\frac{1}{\gamma_n}\sum_{|j|\leq J_n} A_{k-q, j}
  Z_{k-q-j}\Big)-\vep\Big]_+\Big\}\\
   -\exp\Big\{& -\Big[\sum_{k=1}^n
  g\Big(\frac{k}{n},  R_{k,n}+ \frac{1}{\gamma_n}\sum_{|j|\leq J_n} A_{k, j}
  Z_{k-j}\ind\{t=k-j\},\dots, \\ & 
  R_{k-q,n}+\frac{1}{\gamma_n}\sum_{|j|\leq J_n} A_{k-q, j}
  Z_{k-q-j}\ind\{t=k-q-j\} \Big)-\vep\Big]_+\Big\}\Big]
  \\
  \times \ind\, \{ |Z_t| & >\tau 
  \gamma_n,  |Z_s|\leq \tau\gamma_n\ \text{all  $s\in
    C_n, \, s\neq t$}\}\Biggr]\,,
\end{align*}
where, as above, $C_n = \{ -J_n-q+1, \ldots, n+J_n\}$.
Let $L$ be the Lipschitz constant of $g$ with
respect to the metric on $\Eb^q$ given by
\begin{align*}
d((s,x_0,\dots,x_q),(t,y_0,\dots,y_q)) =
|s-t| + \min\{1, |x_0-y_0|+ \dots +|x_q-y_q|\}\}.
\end{align*}
 Notice that, in the obvious notation,
\begin{align}
|\Delta_n|& \leq L\E\sum_{k=1}^n \min \Big\{
1, \sum_{i=0}^q \frac{1}{\gamma_n} \Big| \sum_{|j|\leq J_n} \!\! A_{k-i, j}
          Z_{k-i-j}\ind\{ |Z_{k-i-j}|\leq \tau\gamma_n\}\Big|\Big\} \nonumber \\
&\leq L\,n(q+1)  \E \min \Bigl[
1, \frac{1}{\gamma_n} \Big| \sum_{|j|\leq J_n} \!\! A_{0,j}
          Z_{-j}\ind\{ |Z_{-j}|\leq
          \tau\gamma_n\}\Big|\Bigr]  \label{e:Delta.n.bound}\\
&= L\,n(q+1) \int_0^1 P\Big( \Big| \sum_{|j|\leq J_n} \!\! A_{0,j}
          Z_{-j}\ind\{ |Z_{-j}|\leq \tau\gamma_n\}\Big|>x\gamma_n\Big)\,
      dx\,. \nonumber
\end{align}
Suppose first that $0<\alpha<1$. We have
by \eqref{e:only.large}, as $\tau\downarrow 0$,
\begin{align*}
r_n \Delta_n  \leq o(1) \frac{nr_n}{\gamma_n}
\int_0^{\gamma_n}  P(|Z|> x)\, dx = o(1)nr_n\, P(|Z|> \gamma_n) \to 0
\end{align*}
by Karamata's theorem, and \eqref{Delta.alt} follows.

Consider now the case $\alpha\geq 1$. We abbreviate
\begin{align*}
\Delta_n := E(D_n)= \sum_{t=-J_n-q+1}^{n+J_n} E\bigl(
D_n\ind(B_t)\bigr)\,,
\end{align*}
where
\begin{align*}
B_t = \Bigl\{ |Z_t|>\tau \gamma_n, |Z_s|\leq \tau\gamma_n\ \text{all
  $s\in C_n, \, s\neq t$}\Bigr\}\,.
\end{align*}
Since $g$ has a compact support, there is $\delta>0$ such
that $g(s, x_0,\dots,x_q)=0$ for all $s \in [0,1]$ and
$\{(x_0,\dots,x_q): |x_0|+\dots+|x_q| < \delta\}$. Let $t\in\{ -J_n-q+1,
\ldots ,n+J_n\}$. We have on the event $B_t$,
\begin{align*}
  |D_n|I\{B_t\} \leq
  \Big|\sum_{k=1}^n g\Big(& \frac{k}{n},
  R_{k,n}+\frac{1}{\gamma_n}\sum_{|j|\leq J_n} A_{k, j} Z_{k-j}, \dots, \\ &
  R_{k-q,n}+\frac{1}{\gamma_n}\sum_{|j|\leq J_n} A_{k-q, j}
  Z_{k-q-j}\Big)\\
  -\sum_{k=1}^n g\Big(& \frac{k}{n}, R_{k,n}+
  \frac{1}{\gamma_n}\sum_{|j|\leq J_n} A_{k, j}
  Z_{k-j}\ind\{t=k-j\},\dots, \\ & 
  R_{k-q,n}+\frac{1}{\gamma_n}\sum_{|j|\leq J_n} A_{k-q, j}
  Z_{k-q-j}\ind\{t=k-q-j\}\Big) \Big|I\{B_t\}.
\end{align*}
Let $K_t = \{k : t-J_n \leq k \leq t+q+J_n\}$ and decompose the last
expression into the sum over $K_t$ and $\{1,\dots,n\}\setminus
K_t$. Then, on the event $B_t$, $|\Delta_n|$  is bounded above by
\begin{align}
&  L\, \sum_{k \in K_t}  \min \Bigl\{
  1, \sum_{i=0}^q\frac{1}{\gamma_n} \Big| \sum_{|j|\leq J_n} \!\! A_{k-i, j}
  Z_{k-i-j}\ind\{ |Z_{k-i-j}|\leq \tau\gamma_n\}I\{j \neq
  k-i-t\} \Big|\Bigr\}
  \nonumber \\
  &
  + \|g\|_\infty\, \sum_{k\notin K_t}\, \ind\left\{\sum_{i=0}^q \Bigl(
    \bigl| R_{k-i,n}\bigr| +
  \frac{1}{\gamma_n} \Big| \sum_{|j|\leq J_n} \!\! A_{k-i, j}
  Z_{k-i-j}\ind\{ |Z_{k-i-j}|\leq \tau\gamma_n\}\Big|\Bigr) > \delta\right\}
  \nonumber \\
  & := D_{n,1} + D_{n,2}\,.\label{e:Dn12}
\end{align}

We start with $D_{n,2}$. Recall that for all $i$,  $|R_{k-i,n}|$ is
bounded by $\beta_n \downarrow 0$. Using in the sequel the letter $C$
for a finite positive constant that may change from time to time, we
see that for large $n$,
\begin{align*}
&r_n \sum_{t=-J_n-q+1}^{n+J_n} E\bigl( D_{n,2}\ind(B_t)\bigr)
\leq C\, r_n(n+q+2J_n) n P(|Z|>\tau\gamma_n) \\
& \quad \times P\Big( \sum_{i=0}^q \Big| \sum_{|j|\leq J_n} \!\! A_{k-i,j}
          Z_{k-i-j}\ind\{ |Z_{k-i-j}|\leq
      \tau\gamma_n\}\Big|>\delta\gamma_n/2\Big)\\
&\leq C\, \tau^{-\alpha}n\, (q+1) P\Big( \Big| \sum_{|j|\leq J_n} \!\! A_{k,j}
          Z_{k,j}\ind\{ |Z_{k,j}|\leq
      \tau\gamma_n\}\Big|>\frac{\delta\gamma_n}{2(q+1)}\Big)\,.
\end{align*}
Using \eqref{e:only.large} shows that for small $\tau>0$ this is
further bounded by
$$
C\, n\,P(|Z|>\gamma_n) \to 0 \quad \text{as $n\to\infty$}
$$
 by the choice of $\gamma_n$.

It remains to consider the $D_{n,1}$ term in \eqref{e:Dn12}.
Note that in $D_{n,1}$, for each $t$, $k$ is
restricted to at most $2J_n+q+1$ possible values. We have
\begin{align}
&r_n \sum_{t=-J_n-q+1}^{n+J_n} E\bigl( D_{n,1}\ind(B_t)\bigr) \nonumber \\
&\quad  \leq C\, r_n(n+q+2J_n)P(|Z|>\tau\gamma_n)  \nonumber \\
&\quad  J_n\, \E \min \Bigl[
1, \sum_{i=0}^q\frac{1}{\gamma_n} \Bigl| \sum_{|j|\leq J_n} \!\! A_{k-i,j}
          Z_{k-i-j}\ind\{ |Z_{k-i-j}|\leq \tau\gamma_n\}\Bigr|\Bigr]
        \nonumber \\
& \quad \leq C(q+1)
\frac{J_n}{\gamma_n}\frac{P(|Z|>\tau\gamma_n)}{P(|Z|>\gamma_n)}
\E \Bigl| \sum_{|j|\leq J_n} \!\! A_{0,j}
          Z_{-j}\ind\{ |Z_{-j}|\leq \tau\gamma_n\}\Bigr|\,.\label{e:D.n1}
\end{align}

Suppose first that $\alpha=1$. For large $n$
the last expression can be bounded by
$$
C\, \tau^{-1}\frac{J_n}{\gamma_n}
E\bigl( |Z| \ind\{ |Z| \leq  \tau\gamma_n\}\bigr)\E \sum_{|j|\leq
  J_n} \!\! |A_{j}| \,.
$$
Note that $\E \sum_{|j|\leq   J_n} |A_{j}|$ stays
bounded by \eqref{cc1}. Furthermore, the function
$l(x) = E\bigl( |Z|\ind\{ |Z|\leq x\}\bigr)$ is slowly varying. Therefore,
the above expression vanishes as $n\to\infty$ by \eqref{e:Jn}.

Next consider the case $\alpha>1$. Let $\mu_n= E\bigl( Z\ind\{ |Z|\leq
\tau\gamma_n\}\bigr)$. Note that
\begin{align*}
\E\Big|\!\sum_{|j|\leq J_n} \!\! A_{j}
          Z_{j}\ind\{ |Z_{j}|\!\leq\! \tau\gamma_n\}\Big|
& \leq \E \Big|\! \sum_{|j|\leq J_n} \!\! A_{j}
          \bigl( Z_{j}\ind\{ |Z_{j}|\leq \tau\gamma_n\}-\mu_n\bigr)\Big|
\!+\! |\mu_n|\E \Big | \!\sum_{|j|\leq J_n} \!\! A_{j} \Big |\\
& =: I + II.
\end{align*}
Let us start with $II$. Since $EZ=0$, we see that, as $n\to\infty$,
$$
|\mu_n| \leq E\bigl( |Z| \ind\{ |Z|> \tau\gamma_n\}\bigr)
\sim C\, \tau\, \gamma_n\, P(|Z|>\tau\gamma_n)\,.
$$
Furthermore, to deal with $\sum_{|j|\leq J_n} \!\! A_{j}$, we use the
assumptions \eqref{cc1} -- \eqref{cc2}. Suppose, for example, that
$1<\alpha \leq 2$. Choose $\vep$ small enough so that $\alpha-\vep>1$,
and notice that
$$
\E \bigl| \sum_{|j|\leq J_n} \!\! A_{j}\bigr|
\leq C\, J_n^{1-(\alpha-\vep)^{-1}}
\E  \left( \sum_{|j|\leq J_n} \!\!
\bigl|A_{j}\bigr|^{\alpha-\vep}\right)^{1/(\alpha-\vep)}
\leq C\, J_n^{1-(\alpha-\vep)^{-1}}\,,
$$
and so the corresponding term in \eqref{e:D.n1}
is bounded, for large $n$,  by
$$
C\, \tau^{1-2\alpha}J_n^{2-(\alpha-\vep)^{-1}}P(|Z|>\gamma_n)\,.
$$
Note that for $\vep$ small enough, $\theta:=
2-(\alpha-\vep)^{-1}<\alpha$. In that case the above expression is
$o(J_n/\gamma_n)\to 0$
as $n\to\infty$ by \eqref{e:Jn}. Similarly, in the case $\alpha>2$
this term goes to zero as well.

For $I$, by the Burkholder-Davis-Gundy  inequality,
\begin{align*}
&\E\Big|\!\sum_{|j|\leq J_n} \!\! A_{j} \big(\! Z_{j}\ind\{
|Z_{j}| \!\leq\! \tau\gamma_n\!\}\!-\!\mu_n\!\big)\Big|
\leq C E_A\Big[E_Z \Big( \!\sum_{|j|\leq J_n} \!\! A_{j}^2 \big(
  Z_{j}\ind\{ |Z_{j}| \!\leq\!
  \tau\gamma_n\}\!-\!\mu_n\big)^2\Big)^{\frac12}\Big].
\end{align*}
We use, once again, the assumptions \eqref{cc1}--\eqref{cc2}. Assuming
again that $1<\alpha\leq 2$, and choosing $\vep$
as above, we see that the above expression is bounded by
\begin{align*}
& C\, E_A\Big[ E_Z \Big( \sum_{|j|\leq J_n} \!\! |A_{j}|^{\alpha-\vep}
  \big|
  Z_{j}\ind\{ |Z_{j}|\leq
  \tau\gamma_n\}-\mu_n\big|^{\alpha-\vep}\Big)^{1/(\alpha-\vep)}\Big] \\
& \leq C\, \Big( E\big| Z\ind\{ |Z|\leq
  \tau\gamma_n\}-\mu_n\big|^{\alpha-\vep}\Big)^{1/(\alpha-\vep)}
E\Big( \sum_{|j|\leq J_n} \!\!
|A_{j}|^{\alpha-\vep}\Big)^{1/(\alpha-\vep)}\,,
\end{align*}
which is bounded, and so the corresponding term term in \eqref{e:D.n1}
converges to zero because $J_n/\gamma_n\to 0$ as
$n\to\infty$. The case $\alpha>2$ is entirely analogous (and
simpler).
This completes the proof of the lemma in all cases.
\end{proof}

\section{First Applications}

Theorem \ref{mt} provides a rather complete description of the
asymptotics of the probability of rare events for the sequence
$(X_k)$. In this section we provide some immediate applications of
this theorem. For the sake of simplicity and to avoid complicated formulas we
restrict attention to the case where
both $A_{k,j}$ and $Z_j$ are univariate and $A_{k,j} \geq 0$
a.s. Then $Z$ has a univariate regularly varying distribution and its
limiting measure can be written as in \eqref{mu}.

\subsection{Order statistics}
The first application is to order statistics.
 Let $X_{i:n}$ be the $i$th order statistic of $X_1, \dots, X_n$ in
 descending order. That is,
\begin{align*}
  X_{1:n} \geq X_{2:n} \geq \dots \geq X_{n:n}.
\end{align*}
Fix an integer $q \geq 1$ and consider the $q$-dimensional vector
$(X_{1:n}, \dots, X_{q:n})$ consisting of the $q$ largest values.
We denote by $A^*_i$ the $i$th order statistic of the sequence
$\{A_{j,j}, j\in \Z\}$ in descending order; under the assumptions \eqref{cc1} --
\eqref{cc2} this is a well defined random variable.  Note that for
$\infty > u_1 > u_2> \dots > u_q > 0$, we can write
\begin{align*}
  P(X_{1:n} > \gamma_n u_1, \dots, X_{q:n} > \gamma_n
      u_q) = P(N_n^0 \in B(u_1, \dots, u_q))
\end{align*}
with
\begin{align*}
    B = B(u_1, \dots, u_q) = \cap_{i=1}^q \{\xi: \xi([0,1]
    \times (u_i,\infty)) \geq i\}.
  \end{align*}
Then we have the following implication
of Theorem \ref{mt}.
\begin{cor}\label{os}
Let $d=p=1$, and assume that $A_{k,j} \geq 0$ for all $k,j$, and that
the hypotheses of Theorem \ref{mt} hold. Then
  \begin{align*}
    \frac{P((X_{1:n} > \gamma_n u_1, \dots, X_{q:n} > \gamma_n
      u_q))}{n\Prob(|Z|>\gamma_n)} \to w\, \E \min_{i=1,\ldots, q} \bigl(
A_i^*u_i^{-1}\bigr)^\alpha\,.
  \end{align*}
\end{cor}
\begin{proof}
First note that
  \begin{align*}
    m^0(B(u_1, \dots, u_q)) &= E[(\Leb \times \mu)\circ
    T_{\Abold}^{-1}(B(u_1, \dots, u_q))] \\
    & = E\Big[\mu\Big\{z: \sum \delta_{A_{j,j}z}(u_1,\infty) \geq 1, \dots, \sum
    \delta_{A_{j,j} z}(u_q,\infty) \geq q \Big\} \Big]\\
    & = E[\mu\{z: A^*_1 z \in (u_1, \infty), \dots, A^*_q z \in (u_q,
    \infty)\}]\\
    &= w\, \E \min_{i=1,\ldots, q} \bigl(
A_i^*u_i^{-1}\bigr)^\alpha\,.
  \end{align*}
The claim, therefore, is a direct application  of Theorem \ref{mt}
once we show that the set $ B(u_1, \dots, u_q)$ is bounded away from
the null measure and  $m(\partial B(u_1, \dots,u_q)) = 0$.

The set $B(u_1, \dots, u_q)$ is open. To see this,
write $B(u_1, \dots, u_q) = \cap_{i=1}^q B_i$ with $B_i = \{\xi :
\xi([0,1] \times (u_i,\infty)) \geq i\}$. Then, $B_i^c = \{\xi :
\xi([0,1] \times (u_i,\infty)) < i\}$ and for a
sequence of measures $(\xi_n) \subset B_i^c$ with $\xi_n \vague \xi$ we
have, by the Portmanteau theorem,
\begin{align*}
  i > \liminf_{n \to \infty} \xi_n([0,1] \times
  (u_i,\infty)) \geq \xi([0,1] \times
  (u_i,\infty)).
\end{align*}
Hence, $\xi \in B_i^c$ so $B_i^c$ is closed. This shows that $B_i$ is
open and, consequently, $B(u_1, \dots, u_q)$ is open.
Similarly, the set $C_i = \{\xi :
\xi([0,1] \times [u_i,\infty)) \geq i\}$ is
closed. Since $B_i\subset C_i$ and $C_i$ does not contain the null
measure, we see that each $B_i$ is bounded away from the null
measure and, hence, so is $B$.

Further,   it follows, by the above calculation, that
  \begin{align*}
    m(\partial B(u_1, \dots , u_q)) &= m\bigl({\overline{B(u_1,
    \dots, u_q)}}\bigr) -     m(B(u_1, \dots, u_q)) \\
    &\leq m(\cap_{i=1}^q C_i) - m(B(u_1, \dots, u_q)) \\
    &= E\Big[\int_0^\infty
    I_{[\frac{u_1}{A^*_1}, \infty)}(z) \cdots I_{[\frac{u_q}{A_q^*},\infty)}(z)
    w \,\alpha z^{-\alpha - 1}dz\Big] \\
    &\quad - E\Big[\int_0^\infty
    I_{(\frac{u_1}{A^*_1}, \infty)}(z) \cdots I_{(\frac{u_q}{A_q^*},\infty)}(z)
    w \,\alpha z^{-\alpha - 1}dz\Big] = 0.
  \end{align*}
  This proves the claim.
\end{proof}

\subsection{Hitting times}
Next we consider the large deviations of first hitting times. Take $a
> 0$ and  consider the first hitting time of $(a\gamma_n, \infty)$;
\begin{align*}
  \tau_n = \inf\{k: X_k > a \gamma_n\}.
\end{align*}
\begin{cor}
Let $d=p=1$, and assume that $A_{k,j} \geq 0$ for all $k,j$, and that
the hypotheses of Theorem \ref{mt} hold. Then for any $\lambda>0$
  \begin{align*}
    \frac{P(\tau_n \leq \lambda n)}{n\Prob(|Z|>\gamma_n)}
    \to \lambda E[(A^*_1)^\alpha]
    w a^{-\alpha}
  \end{align*}
\end{cor}
\begin{proof}
It is enough to prove the statement for $\lambda=1$; the proof for a
general $\lambda>0$ will then follow via denoting $m=[\lambda n]$ and
redefining appropriately the sequence $(\gamma_n)$. We
have
\begin{align*}
  r_nP(\tau_n \leq  n) &= r_nP(\sup_{0\leq k \leq n}
  X_k/\gamma_n > a)= r_nP(X_{1:n}>\gamma_na)\,,
\end{align*}
and the statement follows from Corollary \ref{os}.
\end{proof}

\section{Large deviations of the partial sums} \label{sec:ps}

In this section large deviation results for the partial sums
$S_n = X_1 + \dots + X_n$, $n=1,2\ldots$ are considered. The main idea
is to start from Theorem \ref{mt} and derive results for the partial
sum by summing up the points in the point measure $N_n$, while
applying  the continuous mapping argument.

It turns out that for success of this program additional assumptions
are needed. The first assumption is designed to control
the contribution of ``relatively small'' values of the
$X_k$'s. To this end
we introduce the following condition: for each $\delta > 0$
\begin{align}\label{smalljumpcondition}
  \lim_{\vep \downarrow 0}\limsup_n r_n \Prob\Big(\Big|\sum_{k=1}^n
  X_kI\{|X_k| \leq \gamma_n \vep\}\Big| > \gamma_n \delta\Big) =
  0.
\end{align}
The second assumption we need is
\begin{equation} \label{e:summ.A}
\text{the sum $\sum  A_{j,j}$ converges a.s. and} \
 \E \sup_{J\subset \Z}
\Big\| \sum_{j\in J}  A_{j,j} \Big\|^\alpha<\infty\,.
\end{equation}

\begin{thm}\label{partialsumthm} Assume the hypotheses of Theorem
  \ref{mt} and, in addition, that \eqref{smalljumpcondition} and
  \eqref{e:summ.A} hold.
  Then
  \begin{align}\label{partialsum}
    r_n\Prob(\gamma_n^{-1} S_n \in \cdot\,) \to \E\Big[\mu\Big(z:
    \sum_{j \in \Z} A_{j,j}z \in \cdot\,\Big)\Big]
  \end{align}
in $\bfM_0(\R^d )$.
\end{thm}

\begin{rem}\label{unif}
  Note that the large deviation result is uniform in the sense that
  the normalization $r_n$ is the same for all sets. In
  particular, the univariate result ($p = d = 1$) can be stated as
  \begin{align*}
    \lim_{n\to\infty}\frac{\Prob(S_n
      > \rho\gamma_n)}{n\Prob(|Z|>\rho\gamma_n)} = w
    \E\Big(\Big[\big(\sum
    A_{j,j}\big)^+\Big]^\alpha\Big) +(1-w)\E\Big(\Big[\big(\sum
    A_{j,j}\big)^-\Big]^\alpha\Big)
  \end{align*}
  for every $\rho> 0$,  where  the limiting measure associated with
  $Z$ is given by \eqref{mu}.
\end{rem}

\begin{rem}\label{abs}
  In some cases, replacing conditions \eqref{smalljumpcondition} and
  \eqref{e:summ.A} by somewhat stronger conditions, we can modify the
  proof of Theorem \ref{partialsumthm} to obtain large deviations of the
  partial sum of the absolute values of the process. It is sufficient to
  change condition \eqref{e:summ.A} to
  \begin{equation} \label{e:summ.Aabs}
    \E \Big(\sum  \| A_{j,j} \|\Big)^\alpha<\infty\,.
  \end{equation}
  If $0<\alpha\leq 1$, or $\alpha>1$ and $n/\gamma_n\to 0$, then
  it is sufficient to change condition \eqref{smalljumpcondition}
  to, for each $\delta > 0$,
  \begin{align}\label{smalljumpcondition.abs1}
    \lim_{\vep \downarrow 0}\limsup_n r_n \Prob\Big(\sum_{k=1}^n
    |X_k|I\{|X_k| \leq \gamma_n \vep\} > \gamma_n \delta\Big) =
    0.
  \end{align}
  In this case one concludes that
  $S^\text{\text{abs}}_n = \sum_{k=1}^n |X_k|$ satisfies
  \begin{align}\label{partialsumabs}
    r_n\Prob(\gamma_n^{-1} S^{\text{abs}}_n \in \cdot\,) \to \E\Big[\mu\Big(z:
    \sum_{j \in \Z} |A_{j,j}z| \in \cdot\,\Big)\Big]
  \end{align}
  in $\bfM_0(\R^d )$.
  
  If, on the other hand, $\alpha>1$ and $\gamma_n\equiv n$, then
  it is sufficient to change condition \eqref{smalljumpcondition} 
  to, for each $\delta > 0$, 
  \begin{align}\label{smalljumpcondition.abs2}
    \lim_{\vep \downarrow 0}\limsup_n r_n \Prob\Big(\bigl|\sum_{k=1}^n
    \bigl(|X_k|-E|X_0|\bigr)I\{|X_k| \leq n \vep\}\bigr| >
    n \delta\Big) =   0,
  \end{align}
  and then $\bigl( S^{\text{abs}}_n -n E|X_0|\bigr)$ satisfies
  \begin{align}\label{partialsumabs.1}
    r_n\Prob(n^{-1} \bigl( S^{\text{abs}}_n -n E|X_0|\bigr)\in
    \cdot\,) \to \E\Big[\mu\Big(z: \sum_{j \in \Z} |A_{j,j}z| \in
    \cdot\,\Big)\Big]
  \end{align}
  in $\bfM_0(\R^d )$.
\end{rem}

\begin{proof}
  The idea is to divide $S_n$ into three parts. One term
  containing  the terms where $\vep<|X_k| \leq 1/\vep$ for a small
  positive $\vep$,  and the other two parts containing
  terms with $|X_k| \leq \vep$ and $|X_k| >1/\vep$, respetively. The
  contribution from the latter two parts turns out to be negligible.

  For $0<\vep <1$ let $g_\vep$ be a function $[0,1]
  \times \R^d \setminus \{0\} \to \R^d$
  such that $g_\vep(t, x) = g_\vep(x) = x$
  on $\vep<|x| \leq 1/\vep$, $g_\vep(x) = 0$  for all other values of
  $x$. First we consider the convergence of $r_n\Prob(N_n^0(g_{\vep})
  \in \cdot)$ with $N_n^0$ as in Theorem \ref{mt}. Let $m^0$ be the
  limiting measure in \eqref{limitmeasure} with
  $q = 0$. Note that $g_\vep$ is continuous except at the points
  $|x| = \vep$ and $1/\vep$. By Remark \ref{rk:cts.on.sphere}
  \begin{align*}
    m^0\Bigl\{\xi: \xi\bigl([0,1]\times \{|x| = \vep \text{ or }
    1/\vep\}\bigr) > 0\Bigr\} = 0\,.
  \end{align*}
  Hence, the map $\xi \mapsto \xi(g_\vep)$ from $\Np$ to $\R^d$
  satisfies the continuity assumption in the mapping theorem (Lemma
  \ref{mapthm}). Therefore, Theorem \ref{mt}, with $q = 0$, together
  with the  mapping   theorem, imply that
  \begin{align}
    r_n\Prob(N_n^0(g_{\vep}) \in \cdot\,) &\to
      \E\Big[(\Leb \times \mu)\Big((t,z):
      \sum_{j \in \Z} g_{\vep}(t,A_{j,j}z)
      \in \cdot\,\Big)\Big]
      := \tilde m_\vep(\cdot) \label{r1}
  \end{align}
  in $\bfM_0(\R^d)$.

  Put $\tilde m_n(\cdot) = r_n\Prob(\gamma_n^{-1} S_n \in \cdot\,)$
  and $\tilde m$ as in the
  right-hand-side of \eqref{partialsum}. We need to show $ \tilde m_n(f) \to
  \tilde m(f)$ for any $f \in C_0(\R^d)$; in fact, it is sufficient to
  consider uniformly continuous $f$ (see the Appendix). For any such
  $f$ there is  $\eta > 
  0$ such that $x \in
  \textrm{support}(f)$ implies $|x| > \eta$. For any $\delta>0$
  \begin{align*}
    \tilde m_n(f) &= r_n \E[f(\gamma_n^{-1}S_n)] \\
    &= r_n \E[f(\gamma_n^{-1}S_n)I\{|\gamma_n^{-1}S_n-N_n^0(g_\vep)| >
    \delta\}] \\
    &\quad + r_n \E[f(\gamma_n^{-1}
    S_n)I\{|\gamma_n^{-1}S_n-N_n^0(g_\vep)| \leq \delta\}].
  \end{align*}
  The first term is bounded above by
  \begin{align}
    |f|_\infty r_n \Prob(|\gamma_n^{-1} S_n-N_n^0(g_\vep)| &> \delta)
    \leq
    |f|_\infty r_n \Prob\Big(\Big | \sum_{k=1}^n X_kI\{|X_k|\leq
    \gamma_n \vep\}\Big| > \frac{\gamma_n \delta}{2}\Big) \nonumber \\
    & + |f|_\infty r_n \Prob\Big(\Big| \sum_{k=1}^n X_kI\Big\{|X_k|>
    \frac{\gamma_n}{\vep}\Big\}\Big| > \frac{\gamma_n
      \delta}{2}\Big)\,. \label{r2}
  \end{align}
  The assumption \eqref{smalljumpcondition} guarantees that the first
  member in the right hands side of \eqref{r2} is asymptotically
  negligible. The second member in the
  right hands side of \eqref{r2} is, up to a constant, bounded above by
  \begin{align*}
    r_n\Prob\Big(\max_{k=1,\dots,n} |X_k| > \gamma_n/\vep\Big) \leq r_n n
    \Prob(|X_0| > \gamma_n/\vep) \to 0
  \end{align*}
  as first $n \to \infty$ and then $\vep \to 0$.

  Therefore, the statement of the theorem will follow once we show that
  \begin{align} \label{cl1}
   \lim_{\delta  \downarrow 0} \limsup_{\vep \downarrow 0} \limsup_{n \to
      \infty} r_n
    \E\big[\big|f(\gamma_n^{-1}S_n)-f(N_n^0(g_\vep))\big|
      I\{|\gamma_n^{-1}S_n-N_n^0(g_\vep)| \leq  \delta\}\big] = 0,
  \end{align}
  and
  \begin{align}\label{cl2}
    \lim_{\vep \downarrow 0} \tilde m_\vep(f) = \tilde m(f).
  \end{align}
  Indeed, in that case we could write for each $\vep > 0$ and $\delta > 0$
  \begin{align*}
  |\tilde m_n(f)-\tilde m(f)| &\leq
    r_n \E[|f(\gamma_n^{-1}S_n)-f(N_n^0(g_\vep))|I\{|\gamma_n^{-1}S_n-N_n^0(g_\vep)|
    \leq \delta\}] \\ 
    & + r_n \E[f(\gamma_n^{-1}S_n)I\{|\gamma_n^{-1}S_n-N_n^0(g_\vep)|
    > \delta\}]\\
    & + |r_n\E[f(N_n^0(g_\vep))I\{|\gamma_n^{-1} S_n-N_n^0(g_\vep)| \leq
    \delta\}] - r_n\E[f(N_n^0(g_\vep))]| \\
    & + |r_n\E[f(N_n^0(g_\vep))] - m_\vep(f)|\\
    & + |\tilde m_{\vep}(f) - \tilde m(f)|.
  \end{align*}
  By \eqref{cl1}, the argument in
  \eqref{r2}, \eqref{r1}, and \eqref{cl2}, each term converges to $0$
  as first $n \to \infty$, then $\vep \downarrow 0$, and finally
  $\delta \downarrow 0$ .

   It remains to prove \eqref{cl1} and \eqref{cl2}.
   We start with \eqref{cl1}. Choose
   $\delta$ above to be smaller than $\eta/2$. The reason for this is
   that if either $f(\gamma_n^{-1}S_n) > 0$ or $f(N_n^0(g_\vep)) > 0$, then
   on $\{|\gamma_n^{-1} S_n - N_n^0(g_\vep)| \leq \delta\}$ we have
   $|N_n^0(g_\vep)| >
   \eta /2$.  Since $f$ is uniformly continuous the
   expression in \eqref{cl1}
   is bounded above by
   \begin{align*}
     & o_\delta(1)  \,  r_n \E\Bigl[ I\{|N_n^0(g_\vep)| >
     \eta/2\}I\{|\gamma_n^{-1}S_n - N_n^0(g_\vep)| \leq \delta\}\Bigr]\\
     &\quad \leq o_\delta(1) r_n \Prob(|N_n^0(g_\vep)| > \eta/2).
   \end{align*}
   As $n \to \infty$ and $\vep\downarrow 0$, \eqref{r1} and
   \eqref{cl2} (still to be proved) show that this remains bounded by
   const $o_\delta(1)$. As $\delta \downarrow 0$ this converges to $0$.

   It remains to show \eqref{cl2}.
   We have, as $\vep \downarrow 0$,
   \begin{align*}
     \tilde m_\vep(f) &= \int_\Omega \int_{\R^d \setminus \{0\}} f\Big(\sum
     g_\vep(A_{j,j}z)\Big) \mu(dz)\Prob(d\omega)\\
     &\to \int_\Omega \int_{\R^d \setminus \{0\}} f\Big(\sum
     A_{j,j}z \Big) \mu(dz)\Prob(d\omega) \\
     &= \tilde m(f),
   \end{align*}
   by dominated convergence. Indeed, $\sum g_\vep(A_{j,j}z) \to \sum
   A_{j,j}z$, $\mu \times \Prob$-a.e.~as $\vep
   \downarrow 0$, $f$ is continuous, and
$$
\bigl| f(\sum g_\vep(A_{j,j}z))\bigr| \leq |f|_\infty
I\Bigl\{\sup_{J\subset \Z} \Bigl|
\sum_{j\in J}  A_{j,j}z\Bigr| > \eta\Bigr\}\,,
$$
 which is $\mu \times \Prob$-integrable by the
   scaling property of the measure $\mu$ and the assumption
   \eqref{e:summ.A}.
 \end{proof}

\subsection{Checking the conditions of Theorem \ref{partialsumthm}}
To apply Theorem \ref{partialsumthm} one needs to verify the extra
assumptions imposed there. In this section we provide conditions
that are easier to check for some more specific models.

\begin{prop} \label{pr:alpha.le1}
  Let $(X_k)$ be the stationary process in \eqref{x} satisfying the
  conditions of Theorem \ref{mt}.  If $0 < \alpha < 1$, then
  \eqref{smalljumpcondition.abs1} holds and, hence,
  \eqref{smalljumpcondition} holds as well. If $0 < \alpha \leq 1$, then
\eqref{e:summ.Aabs} holds and, hence, \eqref{e:summ.A} holds as well.
\end{prop}

\begin{proof}
  Assume that $0<\alpha<1$. By Markov's inequality, Karamata's
  theorem, and, finally, Theorem \ref{trs}
  \begin{align*}
    \lim_{\vep \downarrow 0}& \limsup_n r_n \Prob\Big(\sum_{k=1}^n
    |X_k|I\{|X_k| \leq \gamma_n \vep\} > \gamma_n \delta\Big) \\
    & \leq \lim_{\vep \downarrow 0}\limsup_n r_n (\gamma_n
    \delta)^{-1}E\Big(\sum_{k=1}^n |X_k|I\{|X_k| \leq \gamma_n
    \vep\}\Big)\\
    &= \lim_{\vep \downarrow 0}\limsup_n r_n n (\gamma_n
    \delta)^{-1}E |X_0|I\{|X_0| \leq \gamma_n\vep\}\\
    &= \lim_{\vep \downarrow 0}\limsup_n r_n n (\gamma_n
    \delta)^{-1} C (\gamma_n \vep) \Prob(|X_0| > \gamma_n\vep) \\
    & = \lim_{\vep \downarrow 0} \limsup_n C\, \frac{\vep \Prob(|X_0| >
      \gamma_n\vep)}{\delta \Prob(|Z|>\gamma_n)}\\
    &= \lim_{\vep \downarrow 0} C\, \frac{\vep^{1-\alpha}}{\delta} = 0,
  \end{align*}
  and so \eqref{smalljumpcondition} holds. If $0 < \alpha \leq 1$ then
  by \eqref{cc1} and Cauchy-Schwarz inequality
  \begin{align*}
    \E\left( \sum \| A_{j}\|\right)^\alpha
    &\leq  \sum E\| A_{j}\|^\alpha \\
    &\leq \sum \bigl(\E\| A_{j}\|^{\alpha -\vep}\bigr)^{1/2}
    \bigl(\E\| A_{j}\|^{\alpha +\vep}\bigr)^{1/2}\\
    &\leq \left( \sum E\| A_{j}\|^{\alpha -\vep}\right)^{1/2}
    \left( \sum E\| A_{j}\|^{\alpha +\vep}\right)^{1/2}<\infty,
\end{align*}
and so \eqref{e:summ.Aabs} holds.
\end{proof}

If the sum \eqref{x} defining the process $(X_k)$ is finite, then
modest additional assumptions on the sequence $(\Abold_k)_{k \in \Z}$
will guarantee applicability of Theorem \ref{partialsumthm}. We
present one such situation. 

\begin{prop}\label{pr:finite.sum}
 Let $(X_k)$ be the stationary process in \eqref{x} satisfying the
conditions of Theorem \ref{mt}. Suppose, further, that the sequence
$(\Abold_k)_{k \in \Z}$ is i.i.d. such that for some $M=0,1,2,\ldots$,
$A_{k,j}=0$ a.s. for $|j|>M$. Then \eqref{smalljumpcondition} holds
and, further, \eqref{smalljumpcondition.abs1} and
\eqref{smalljumpcondition.abs2} (as appropriate) hold. Also,
both \eqref{e:summ.Aabs} and
\eqref{e:summ.A} hold, and so Theorem \ref{partialsumthm} applies.
\end{prop}

When the i.i.d.\ assumption of the sequence $(\Abold_k)_{k \in \Z}$ is
dropped one can still obtain sufficient conditions for
\eqref{smalljumpcondition}. See Lemma \ref{l:trunc.low}. 

\begin{proof}
For finite sums the condition \eqref{e:summ.Aabs} is a trivial
consequence of \eqref{cc1} - \eqref{cc2}. We will show that
\eqref{smalljumpcondition} holds; the proof for
\eqref{smalljumpcondition.abs1} and \eqref{smalljumpcondition.abs2} is
similar. It is, clearly, enough to consider the
case $d=1$. Notice, further, that
\begin{equation} \label{e:bound.small1}
\Prob\Big(\bigl|\sum_{k=1}^n
  X_kI\{|X_k| \leq \gamma_n \vep\}\bigr| > \gamma_n \delta\Big)
\end{equation}
$$
\leq \Prob\Big( \text{for some $k=1,\ldots, n$,} \  \
  |A_{k,j}Z_{k-j}|>\gamma_n\vep \ \ \text{for at least 2 different
    $j$}\bigr)
$$
$$
+ \Prob\Big(\bigl|\sum_{k=1}^n
  X_kI\{ |A_{k,j}Z_{k-j}|\leq \gamma_n M\vep \ \text{for all
    $k=1,\ldots ,n$ and $|j|\leq M$}\}>\gamma_n \delta\Big)\,.
$$
The first term in the right hand side of \eqref{e:bound.small1} is
bounded by
$$
n\sum_{{i,j=-M, \ldots, M}\atop{i\neq j}}\Prob\Big(
|A_{0,i}Z_{-i}|>\gamma_n\vep , \, |A_{0,j}Z_{-j}|>\gamma_n\vep
\bigr)
$$
$$
= o(1)\, n\Prob (|Z|>\gamma_n) = o(1)(1/r_n),
$$
as in Lemma 3.4 in \cite{HS08I}. The second term in the right hand
side of \eqref{e:bound.small1} does not exceed
$$
\sum_{j=-M}^M \Prob\Big(\bigl|\sum_{k=1}^n A_{k,j}Z_{k-j}
I\{ |A_{k,j}Z_{k-j}|\leq \gamma_n M\vep \ \text{for all
    $k=1,\ldots ,n$}
$$
$$
\text{and $|j|\leq M$} \}>\gamma_n
\delta/(2M+1)\Big)\,.
$$
By the assumed independence, for every $|j|\leq M$,
$$
  \lim_{\vep \downarrow 0}\limsup_n r_n
\Prob\Big(\bigl|\sum_{k=1}^n A_{k,j}Z_{k-j}
I\{ |A_{k,j}Z_{k-j}|\leq \gamma_n M\vep \ \text{for all
    $k=1,\ldots ,n$}
$$
$$
\text{and $|j|\leq M$} \}>\gamma_n
\delta/(2M+1)\Big) =0\,,
$$
see the argument in Lemma 2.1 in \citep{HLMS05}. Therefore,
\eqref{smalljumpcondition}  follows.
\end{proof}

Finally, for certain symmetric  stochastic recurrence equations as in
Examples  \ref{ex:SRE1} and \ref{ex:SRE} we provide sufficient
conditions for the applicability of Theorem
\ref{partialsumthm}.
\begin{prop}\label{pr:symm.SRE}
Suppose that the i.i.d. pairs $(Y_k, Z_k)_{k \in \Z}$ are symmetric
(i.e. $(-Y_k, -Z_k)\eid (Y_k, Z_k)$), $Z \in \RV(\mu,\alpha)$ for some
$0<\alpha<2$ and
  $E\|Y\|^{\alpha + \vep} < 1$ for some $\vep > 0$. Then the
random recursion \eqref{e:SRE} has a unique stationary solution, and
it satisfies Theorem \ref{partialsumthm}.
\end{prop}
\begin{proof}
Existence and uniqueness of a stationary solution follows from
Corollary 2.3 in \cite{HS08I}, which also shows that this solution is
of the form \eqref{x} and satisfies the assumptions of Theorem
\ref{mt}. For $0<\alpha<1$ the statement follows from Proposition
\ref{pr:alpha.le1}. For $\alpha\geq 1$ we have by convexity (see Lemma
3.3.1 in \cite{kwapien:woyczynski:1992})
$$
 \E\left( \sum \| A_{j}\|\right)^\alpha
=  \E\left( \sum_{j=0}^\infty \| \prod_{i=0}^j Y_i\|\right)^\alpha
$$
$$
\leq  \E\left( \sum_{j=0}^\infty  \prod_{i=0}^j \| Y_i\|\right)^\alpha
\leq \left[ \sum_{j=0}^\infty \left( \E\prod_{i=0}^j \|
  Y_i\|^\alpha\right)^{1/\alpha}\right]^\alpha <\infty
$$
since $E\|Y\|^\alpha < 1$. Therefore, \eqref{e:summ.Aabs} holds. Further,
the symmetry assumption in the proposition guarantees that the
stationary process $(X_k)$ is symmetric in the sense that $(X_k)\eid
(\epsilon_kX_k)$, where $(\epsilon_k)$ is a sequence of
i.i.d. Rademacher random variables independent of $(X_k)$. We conclude
as in the proof of Proposition \ref{pr:alpha.le1}
\begin{align*}
    \lim_{\vep \downarrow 0}& \limsup_n r_n \Prob\Big(\Bigr|\sum_{k=1}^n
    X_kI\{|X_k| \leq \gamma_n \vep\}\Bigl| > \gamma_n \delta\Big) \\
    & \leq \lim_{\vep \downarrow 0}\limsup_n r_n (\gamma_n
    \delta)^{-2}E\Big(\bigl|\sum_{k=1}^n X_kI\{|X_k| \leq \gamma_n
    \vep\}\bigr|\Big)^2\\
     & = \lim_{\vep \downarrow 0}\limsup_n r_n (\gamma_n
    \delta)^{-2}   E\Big(\sum_{k=1}^n X_k^2I\{|X_k| \leq \gamma_n
    \vep\}\Big)\\
    &= \lim_{\vep \downarrow 0}\limsup_n r_n n (\gamma_n
    \delta)^{-2}E X_k^2I\{|X_k| \leq \gamma_n\vep\}\\
   &= \lim_{\vep \downarrow 0}\limsup_n r_n n (\gamma_n
    \delta)^{-2} C (\gamma_n \vep)^2 \Prob(|X_k| > \gamma_n\vep) \\
   & = \lim_{\vep \downarrow 0} \limsup_n C\, \frac{\vep^2 \Prob(|X_k| >
      \gamma_n\vep)}{\delta^2 \Prob(|Z|>\gamma_n)}\\
   &= \lim_{\vep \downarrow 0} C\, \frac{\vep^{2-\alpha}}{\delta^2} = 0,
  \end{align*}
proving \eqref{smalljumpcondition}.

\end{proof}

\section{Ruin probabilities}

In this section we consider the univariate ($d = p = 1$) ruin problem
based on the sequence $(X_k)$ in \eqref{x}. Throughout this section we
assume that $\alpha >
1$ (which requires, according to our assumptions, that $E Z = 0$), and
let $c >0$ be the ``drift''. We are interested in deriving the
asymptotic decay of the so-called infinite horizon ruin probability
\begin{align*}
  \psi(u) = P\Big(\sup_n (S_n - c n) >  u\Big)
\end{align*}
as $u \to \infty$. Here $S_n = X_1 + \dots + X_n$ is the partial sum
process.


As in Section \ref{sec:ps}, we will need to assume extra technical
conditions, mostly in order to control the contributions of the
small jumps to the ruin probability. We start with some notation. For integer
$q\geq -1$ let
\begin{equation} \label{e:trunc.below}
\hat X^q_k =     \sum_{|j|>q} A_{k,j}Z_{k-j}, \, k\in\Z\,.
\end{equation}
We assume that, for each $q \geq -1$ and each $\delta > 0$,
\begin{equation} \label{smalljumpcondition.abs2q}
  \lim_{\vep \downarrow 0}\limsup_n r_n \Prob\Big(\Big|\sum_{k=1}^n
  \bigl(|\hat X^q_k|-E|\hat X^q_0|\bigr)I\{|\hat X^q_k| \leq n \vep\}\Big| >
n \delta\Big) =   0,
\end{equation}
and that  for every  $q \geq 0$ and $\gamma>0$,
\begin{align}\label{smalljumpcondition.abs2q2}
\lim_{\delta\to 0} \limsup_{n\to\infty}
  \frac{P\Big(\sup_{k \leq n}\bigl|\sum_{i=1}^k X^q_iI\{|X^q_i| \leq n
  \delta\}\bigr| > n\gamma\Big)}{nP(|Z|> n)} = 0.
\end{align}
It is easy to check that condition \eqref{smalljumpcondition.abs2q}
holds, for example, under the assumptions of Proposition
\ref{pr:finite.sum}. Sufficient conditions for
\eqref{smalljumpcondition.abs2q2} are given in Lemma \ref{l:trunc.low}
below.

\begin{thm} \label{t:ruin.pr}
Suppose that the
  conditions of Theorem \ref{mt} hold with $\alpha > 1$.  Suppose,
  additionally, that
  \eqref{smalljumpcondition.abs2q}, \eqref{smalljumpcondition.abs2q2},
  and \eqref{e:summ.Aabs} hold. Then 
  \begin{align}\label{univruin}
   & \lim_{u \to \infty} \frac{\psi(u)}{uP(|Z| > u)} = \\
    & \quad \quad\quad E\Big[w \Big(\sup_{j \in \Z} \sum_{k=-\infty}^j
    A_{k,k}\Big)^\alpha + (1-w) \Big(\sup_{j \in \Z}
    \sum_{k=-\infty}^j
    -A_{k,k}\Big)^\alpha\Big] \frac{1}{c(\alpha-1)}. \nonumber
  \end{align}
\end{thm}

\begin{exmp}[iid] In the iid case $A_{k,j} = I\{j=0\}$ and we get the
  classical result (see e.g.\ \cite{EKM97})
  \begin{align*}
    \lim_{u \to \infty} \frac{\psi(u)}{uP(|Z| > u)} =
    \frac{w}{c(\alpha - 1)}.
  \end{align*}
\end{exmp}
 \begin{exmp}[SRE] Consider a univariate SRE of the Examples
 \ref{ex:SRE1} and
      \ref{ex:SRE}. Assume that the i.i.d. pairs $(Y_k, Z_k)_{k \in
        \Z}$ are symmetric, and that \eqref{smalljumpcondition.abs2q}
      holds. Put
    \begin{align*}
     M_+ &= \sup_{j \geq 0}\Big(\sum_{k=0}^j Y_1 \cdots Y_{k}\Big)^+,\\
     M_- &= \sup_{j \geq 0}\Big(\sum_{k=0}^j Y_1 \cdots Y_{k}\Big)^-.
   \end{align*}
 Then
   \begin{align*}
     \lim_{u \to \infty} \frac{\psi(u)}{uP(|Z| > u)}
     &= \Big(w\E[M_+^\alpha] +
     (1-w)\E[M_-^\alpha]\Big)\frac{1}{c(\alpha-1)}\,.
   \end{align*}
 This result is believed to be new.
 \end{exmp}

\begin{proof}
For $q\geq 0$ we define a counterpart to \eqref{e:trunc.below} by
\begin{align*}
  X^q_k = \sum_{|j| \leq q} A_{k,j}Z_{k-j}, \, k\in\Z\,,
\end{align*}
and let
\begin{align*}
  S^q_n = X^q_1 + \dots + X^q_n, \quad \hat S^q_n = \hat X^q_1 +
  \dots + \hat X^q_n,\, n=1,2,\ldots.
\end{align*}
  Let $R$ denote the right-hand-side of \eqref{univruin}.
  The first step is to prove the upper bound
  \begin{align}\label{eq:univub}
    \limsup_{u \to \infty} \frac{\psi(u)}{uP(|Z|>u)} \leq R.
  \end{align}
  For (a large) integer $M=1,2,\dots$, $\psi(u)$ is bounded above by
  \begin{align*}
    &\Prob\Big(\sup_{k \leq [u]M} (S_k - c k) > [u]\Big)
    + \Prob\Big(\sup_{k > [u]M} (S_k - ck)>  [u]\Big) \\
    &=: p^{(11)}_M(u) + p^{(12)}_M(u).
  \end{align*}
  By Lemma \ref{lem1a}
  \begin{align*}
    \lim_{M\to\infty}\limsup_{u \to
      \infty}\frac{p^{(12)}_M(u)}{uP(|Z|>u)} = 0,
  \end{align*}
  so the main contribution comes from $p^{(11)}_M(u)$.
  For any $\vep > 0$ and any integer $q \geq 0$, we have the upper bound
  \begin{align*}
    p^{(11)}_M(u) &\leq \Prob\Big(\sup_{k \leq [u]M} (S_k^q - c k) >
    [u](1-\vep)\Big) + \Prob\Big(\sup_{k \leq [u]M} \hat
    S^q_k > [u] \vep\Big).
  \end{align*}
  It follows from Remark \ref{abs} and assumptions
  \eqref{smalljumpcondition.abs2q} and \eqref{e:summ.Aabs} that
  \begin{align*}
    \lim_{q \to \infty}\lim_{u \to 0}\frac{P(\sup_{k \leq [u]M} \hat S^q_k >
      [u]\vep)}{uP(|Z| > u)} = 0.
  \end{align*}
  It remains to show
  \begin{align} \label{ubnd}
    \lim_{M\to \infty} \lim_{\vep \to 0} \limsup_{q \to \infty} \limsup_{u
      \to \infty} \frac{\Prob\Big(\sup_{k \leq [u]M} (S_k^q - c k) >
    [u](1-\vep)\Big)}{uP(|Z| > u)} \leq R.
  \end{align}
  Putting $n = [u]M$ and taking $0<\gamma<1$, and a small $\delta >
  0$, we see that
  \begin{align}
    &\Prob\Big(\sup_{k \leq [u]M} (S_k^q - c k) >
    [u](1-\vep)\Big) \nonumber \\ & \quad = \Prob\Big(\sup_{k \leq n} n^{-1}(S^q_k
    - ck) > (1-\vep)M^{-1},\sup_{k \leq n}\bigl| \sum_{i=1}^k
    X^q_iI\{|X^q_i| \leq     n \delta\}\bigr| \leq
    n\gamma\Big) \label{ubndmain} \\
    &\quad +
    \Prob\Big(\sup_{k \leq n } n^{-1}(S^q_k - ck) > (1-\vep)M^{-1},
    \sup_{k\leq n}\bigl| \sum_{i=1}^k X^q_iI\{|X^q_i| \leq n
    \delta\}\bigr| > n\gamma\Big). \nonumber
\end{align}
Notice that, by the regular variation and
\eqref{smalljumpcondition.abs2q2},
 for every $M$ (recall $n = [u]M$) and $0<\gamma<1$
\begin{align}\label{eq:nosmalljumps1}
\lim_{\delta\to 0} \limsup_{u\to\infty}
  \frac{P\Big(\sup_{k \leq [u]M}\bigl|\sum_{i=1}^k X^q_iI\{|X^q_i| \leq n
  \delta\}\bigr| > n\gamma\Big)}{uP(|Z|> u)} = 0.
\end{align}
Hence, we are left with estimating \eqref{ubndmain}.



Since each noise variable $Z$ affects at most $2q+1$ values of the
process $( X^q)$, it follows from  the
obvious fact that for every $\delta>0$,
$$
\lim_{n\to\infty}\frac{P\bigl(|Z_j|>n\delta \ \ \text{for at least two
    different   $j=-q,\ldots, n+q$}\bigr)}{nP(|Z|>n)}=0,
$$
and Remark 4.1 in \cite{HS08I}, that
\begin{equation} \label{e:unl.event}
\Prob(D_n) := \Prob\bigl( |X_{j_i}^q|>n\delta \ \ \text{for $j_1,j_2 =
  1,\ldots, n$, $|j_1-j_2|>2q$}\bigr) = o\bigl( nP(|Z|>n)\bigr).
\end{equation}
We conclude by \eqref{e:unl.event} and \eqref{eq:nosmalljumps1} that
for the upper bound we need to prove that
\begin{equation} \label{e:upper.new1}
\lim_{M\to\infty}\lim_{\vep \to 0} \limsup_{q \to \infty}
      \lim_{\gamma\to 0}\limsup_{\delta \to 0}\limsup_{u \to \infty}
      \frac{p_n(\delta)}{uP(|Z|> u)}\leq R\,.
\end{equation}
Here $p_n(\delta)$ is a modification of the probability in
\eqref{ubndmain}, defined as follows.

For $n\geq 1$ and $\delta>0$ we denote
\begin{align*}
K_\delta(n) = \inf\{ i=1,\ldots, n:\, |X_i|>n \delta\}\,,
\end{align*}
defined to be equal to $n+1$ if the infimum is taken over the empty
set. Then we set
\begin{align*}
p_n(\delta) =
\Prob\bigg( \sup_{0\leq k\leq 2q}
n^{-1}\sum_{i=K_\delta(n)}^{K_\delta(n)+k} X_i -c K_\delta(n)>
(1-\vep)M^{-1}-2\gamma\biggr)\,.
\end{align*}
This puts us in a situation where we can use the large deviations for
point processes in Theorem \ref{mt} and the mapping theorem in Lemma
\ref{mapthm}.

Let $q^\prime = 6q+1$. This will correspond to the dimension of the
point processes we will work with. Specifically, $q^\prime$ is
the number of values of the process we are keeping track of in
\eqref{Nn}, and we will use the statement of Theorem \ref{mt} in the
space $\bfM_0(\Np^{q^\prime})$. We define now a functional
$h^\ast:\, \Np^{q^\prime} \to \bbr$ as follows. Let
\begin{align*}
\xi = \sum_{k\in\bbz} \delta_{(t_k,x_k^{(1)},\ldots,
  x_k^{(q^\prime)})} \, \in \Np^{q^\prime}\,.
\end{align*}
Consider all points $\bigl(t_k,x_k^{(1)},\ldots,x_k^{(q^\prime)}\bigr)$
of $\xi$ satisfying the following two conditions:
\begin{enumerate}
\item  for some $m=2q+1,\ldots, q^\prime-2q$, $|x_k^{(m)}|>\delta$;
\item  $|x_k^{(j)}|\leq \delta$ for all $j=1,\ldots, 2q$ and all 
  $j=q^\prime -2q+1,\ldots, q^\prime$.
\end{enumerate}
Note that, by the definition of the space $\Np^{q^\prime}$, the set
$H_\delta(\xi)$ of such points is finite. If $H_\delta(\xi)=\emptyset$, we
set $h^\ast(\xi)=0$.

With the obvious convention for the expression $k\in H_\delta(\xi)$, we
set, for each such $k$,
\begin{align*}
m_k= \min\bigl\{ m=2q+1,\ldots, q^\prime-2q:\, |x_k^{(m)}|>\delta\bigr\}\,,
\end{align*}
and define
\begin{align*}
h\bigl( t_k,x_k^{(1)},\ldots,x_k^{(q^\prime)}\bigr)
= \max_{j=0,1,\ldots, 2q}\Bigl(
\sum_{i=m_k}^{m_k+j}x_k^{(i)}\Bigr)-ct_k\,.
\end{align*}
Finally, we define
\begin{equation} \label{e:funct}
h^\ast(\xi) = \max_{k\in H_\delta(\xi)} h\bigl(
t_k,x_k^{(1)},\ldots,x_k^{(q^\prime)}\bigr)\,.
\end{equation}
It follows from  Lemma \ref{l:funct.cts} that the measure $m^{q^\prime}$
in Theorem \ref{mt} assigns zero value to the set of discontinuities
of $h^\ast$.

By the mapping theorem (Lemma \ref{mapthm}), we conclude that for any
$\tau>0$,
\begin{equation} \label{e:use.conv}
r_n\Prob\bigl( h^\ast( N_n^{q^\prime})>\tau\bigr) =
m_n^{q^\prime}\bigl( h^\ast( N_n^{q^\prime})>\tau\bigr)
\to m^{q^\prime}\bigl\{ \xi:\,  h^\ast( \xi)>\tau\bigr),
\end{equation}
using the fact the right hand side of \eqref{e:use.conv} is
continuous in $\tau>0$.

Taking now into account the definition of $r_n$, the estimate
\eqref{e:unl.event} and the fact that $n=[u]M$, one obtains from
\eqref{e:use.conv} that
$$
\limsup_{u \to \infty} \frac{p_n(\delta)}{uP(|Z|> u)}
\leq M^{-(\alpha-1)} m^{q^\prime}\Bigl\{ \xi:\,  h^\ast( \xi)>
(1-\vep)M^{-1}-2\gamma\Bigr\},.
$$

It follows from the form of the limiting measure $m^{q^\prime}$ in the
one-dimensional case (see \eqref{mu}) that for any $\tau>0$,
$$
\lim_{\delta\to 0} m^{q^\prime}\Bigl\{ \xi:\,  h^\ast(
\xi)>\tau\Bigr\}
= \frac{1}{c(\alpha-1)} \bigl( \tau^{-(\alpha-1)} -
(c+\tau)^{-(\alpha-1)}\bigr)
$$
$$
E\Big[w \Big(\max_{j =0,1,\ldots, 2q}
  \sum_{k=-q}^{-q+j}  A_{k,k}\Big)_+^\alpha + (1-w) \Big(\max_{j
    =0,1,\ldots, 2q } \sum_{k=-q}^{-q+j} -A_{k,k}\Big)_+^\alpha\Big]\,,
$$
from which we see that
$$
\limsup_{q \to \infty}
      \lim_{\gamma\to 0}\limsup_{\delta \to 0}\limsup_{u \to \infty}
      \frac{p_n(\delta)}{uP(|Z|> u)}
$$
$$
\leq  M^{-(\alpha-1)} \bigl[ \bigl((1-\vep)M\bigr)^{\alpha-1} -\bigl(
  c+((1-\vep)M^{-1})\bigr)^{-(\alpha-1)}\bigr] R\,,
$$
from which \eqref{e:upper.new1} follows. This proves the upper bound
\eqref{eq:univub}.

The lower bound requires a similar estimate. Take $\vep > 0$ and let
$u$ be sufficiently large that $([u]+1)/[u] < 1+\vep$.
For (a large) integer $M=1,2,\ldots$ we have
\begin{align*}
  \psi(u) &\geq P(\sup_{k} (S_k - ck) > [u]+1) \\
  &\geq  P(\sup_{k \leq [u]M} (S_k - ck) > [u](1+\vep)) \\
  &\geq  P(\sup_{k \leq [u]M} (S_k^q - ck) > [u](1+2\vep)) \\
  & \quad - P(\sup_{k \leq [u]M} \hat S_k^q  > [u]\vep).
\end{align*}
Hence, by Remark \ref{abs} and assumptions
  \eqref{smalljumpcondition.abs2q} and \eqref{e:summ.Aabs}, it is
  sufficient to prove  that
\begin{align*}
  \lim_{M\to \infty} \lim_{\vep \to 0} \liminf_{q \to \infty}
    \liminf_{u \to \infty} \frac{\Prob\Big(\sup_{k \leq [u]M}
    (S_k^q - c k) > [u](1+2\vep)\Big)}{uP(|Z| > u)} \geq R.
\end{align*}
Using \eqref{eq:nosmalljumps1} again it is sufficient to consider
\begin{align} \label{lbnd}
  \frac{\Prob\Big(\sup_{k \leq [u]M} n^{-1}(\sum_{i=1}^k X_iI\{n \delta <|X_i|<
  n/\delta\} - ck) >
    (1+2\vep)/M+\gamma\Big)}{uP(|Z| > u)},
\end{align}
and the argument from here is the same as in the case of the upper
bound.
\end{proof}

Below are the lemmas used in the proof of Theorem \ref{t:ruin.pr}.
\begin{lem}\label{lem1a} Under the assumptions of Theorem
  \ref{t:ruin.pr}
\begin{align*}
\lim_{M\to\infty} \limsup_{u\to\infty}\dfrac{\Prob(\sup_{k > u\, M}
  (S_k - ck) > u)}{u\, \Prob(|Z| > u)}=0\,.
\end{align*}
\end{lem}

\begin{proof}
  We use  Theorem \ref{partialsumthm} and  Remark \ref{unif} and \ref{abs} with
  $\gamma_n\equiv n$. Choose $\beta>1$ and $A>1$ such that
  \begin{equation} \label{e:beta}
    E|X_1|<\frac{c(1-1/A)}{\beta-1}\,,
  \end{equation}
  and write
  \begin{align*} 
    \Prob(\sup_{k > u\, M}   (S_k - ck) > u)
    \leq \sum_{j=1}^\infty \Prob\Bigl( S_k>ck \ \ \text{for some
      $Mu\beta^{j-1}\leq k\leq Mu\beta^{j}$}\Bigr)\,.
  \end{align*}
  By stationarity of $(X_k)$, for every $j=1,2,\ldots$,
  \begin{align*}
    \Prob\Big( S_k>ck \ \ \text{for some
      $Mu\beta^{j-1}\leq k\leq Mu\beta^{j}$}\Big)
    \leq \Prob\Big( S_{\lceil Mu\beta^{j-1}\rceil}>
    \frac{cMu\beta^{j-1}}{A}\Big)& \\
    + \Prob\Big( S_k>ck + cMu\beta^{j-1}(1-1/A)
    \ \ \text{for some
      $0\leq k\leq Mu\bigl(\beta^{j} - \beta^{j-1}\bigr)$}\Big)&\,.
  \end{align*}
  Using Theorem \ref{partialsumthm} 
  we see that for some positive constant
  $C$ (that, as usual, may change in the sequel) we have, for $u$
  large enough,
  \begin{align*}
  \Prob\Big( S_{\lceil Mu\beta^{j-1}\rceil}>
  \frac{cMu\beta^{j-1}}{A}\Big)
  \leq C Mu\beta^{j-1} P(|Z| > Mu\beta^{j-1})
  \end{align*}
  and, by Potter's bound,  for $M$ large enough,
  \begin{align*}
    \frac{  P(|Z| > Mu \beta^{j-1})}{P(|Z| > u)} \leq C (M\beta^{j-1})^{-\alpha}\,.
  \end{align*}
  It follows that
  \begin{align*}
    \limsup_{u\to\infty}\dfrac{\sum_{j=1}^\infty \Prob\Big( S_{\lceil
        Mu\beta^{j-1}\rceil}> \frac{cMu\beta^{j-1}}{A}\Big)}{u\, \Prob(|Z| >
      u)}
    \leq CM^{-(\alpha-1)}\,.
  \end{align*}
  Using the fact that
  $\alpha>1$, we let $M\to\infty$ and see that the above expression
  converges to zero.
  
  Furthermore, for every $j=1,2,\ldots$,
  \begin{align*}
    & \Prob\Big( S_k>ck + cMu\beta^{j-1}(1-1/A)
    \ \ \text{for some
      $0\leq k\leq Mu\bigl(\beta^{j} - \beta^{j-1}\bigr)$}\Big)\\
    & \quad \leq \Prob\Big( \sum_{k=0}^{Mu(\beta^{j} -
      \beta^{j-1})}|X_k|>cMu\beta^{j-1}(1-1/A)\Big) \\
    & \quad \leq \Prob\Bigl(\frac{1}{Mu(\beta^{j} -
      \beta^{j-1})} \!\!\!\!\! \sum_{k=0}^{Mu(\beta^{j} -
      \beta^{j-1})} \!\!\!\!\!(|X_k|-E|X_1|) >
    \frac{c(1-1/A)}{\beta-1}- E|X_1|\Bigr).
    \end{align*}
  By the choice of $\beta$ and $A$ as in \eqref{e:beta} and the
  assuption \eqref{smalljumpcondition.abs2q}, we can use the
  large deviations result \eqref{partialsumabs.1}, to conclude that,
  just as above,
  for all $M$ large enough,
  \begin{align*}
    &\limsup_{u \to \infty} \frac{\Prob\Bigl( S_k>ck + cMu\beta^{j-1}(1-1/A)
    \ \ \text{for some
      $0\leq k\leq Mu\bigl(\beta^{j} -
      \beta^{j-1}\bigr)$}\Bigr)}{uP(|Z|>u)}\\ 
  & \quad \leq C\bigl( M \beta^{j-1}\bigr)^{-(\alpha-1)}\,,
\end{align*}
  and, as before, these bounds can be summed up over $j$ and, then, one lets
  $M\to \infty$  and uses the fact that $\alpha>1$. This proves the
  statement of the lemma.
\end{proof}

\begin{lem} \label{l:funct.cts}
Under the assumptions of Theorem   \ref{t:ruin.pr}, the measure
$m^{q^\prime}$ in Theorem \ref{mt}  does not charge the  set of
discontinuities of $h^\ast$ in \eqref{e:funct}.
\end{lem}
\begin{proof}
Let $\Xi$ be the subset of $\Np^{q^\prime}$ consisting of point
measures $\xi$ such that
$$
\xi \big([0,1] \times \big\{(x_0,\dots,x_{q^\prime}):\,
|x_i|=\delta, \, \text{ some }  i\in\{0,\dots,q^\prime\}\big\}
\big)=0\,.
$$
According to Remark \ref{rk:cts.on.sphere}, the  measure
$m^{q^\prime}$ is concentrated on the set $\Xi$, and so it is enough
to prove that the functional $h^\ast$ is continuous at each
$\xi\in\Xi$. Let $(\xi_n)$ be a sequence in $\Np^{q^\prime}$ such that
$\xi_n \vague \xi$. If $H_\delta(\xi)=\emptyset$, then
$H_\delta(\xi_n)=\emptyset$ for all $n$ large enough, and so
$h^\ast(\xi_n)=0\to 0=h^\ast(\xi)$.

Suppose now that $H_\delta(\xi)\not=\emptyset$. By the definition of
the set $\Xi$ we see that for all $n$ large enough (say, $n\geq n_0$),
the cardinality of
$H_\delta(\xi_n)$ is equal the (finite) cardinality of
$H_\delta(\xi)$. Moreover, the vague convergence $\xi_n \vague \xi$
implies that, for every $n\geq n_0$  there is an enumeration
$\bigl\{\bigl( (t_k)^n,(x_k^{(1)})^n,\ldots,(x_k^{(q^\prime)})^n\bigr)\bigr\}$
of $H_\delta(\xi_n)$ such that for every $k\in H_\delta(\xi)$,
$$
\bigl( (t_k)^n,(x_k^{(1)})^n,\ldots,(x_k^{(q^\prime)})^n\bigr)
\to \bigl( t_k,x_k^{(1)},\ldots,x_k^{(q^\prime)}\bigr)
$$
componentwise as $n\to\infty$ (see \cite{R87}). Therefore, for each
such $k$,
$$
h\bigl( (t_k)^n,(x_k^{(1)})^n,\ldots,(x_k^{(q^\prime)})^n\bigr)
\to h\bigl( t_k,x_k^{(1)},\ldots,x_k^{(q^\prime)}\bigr),
$$
and, since the set $H_\delta(\xi)$ is finite, we conclude that
$h^\ast(\xi_n)\to h^\ast(\xi)$, as required.
\end{proof}

Finally, as promised, we provide sufficient conditions for 
\eqref{smalljumpcondition.abs2q2}. 

\begin{lem} \label{l:trunc.low}
Assume the hypothesis of Theorem  \ref{mt}. If $1 < \alpha \leq  2$,
then for every  $q \geq 0$ and $\gamma>0$,
\begin{align*}
\lim_{\delta\to 0} \limsup_{n\to\infty}
  \frac{P\Big(\sup_{k \leq n}\bigl|\sum_{i=1}^k X^q_iI\{|X^q_i| \leq n
  \delta\}\bigr| > n\gamma\Big)}{nP(|Z|> n)} = 0.
\end{align*}
If $\alpha > 2$ the conclusion holds if 
additionally, for some
$\beta > \alpha - 1$ and all $-q \leq j \leq q$,  
\begin{align}\label{A.additional.assumption}
  E A_{0,j}^{2\beta} < \infty.
\end{align}  
\end{lem}
\begin{proof}
Write
\begin{align*}
  & P\Big(\sup_{k \leq n}\Big|\sum_{i=1}^k X^q_iI\{|X^q_i| \leq n
  \delta\}\Big| > n\gamma\Big)\\
  & \quad \leq \sum_{|j|\leq q} P\Big(\sup_{k \leq n}\Big|\sum_{i=1}^k
    A_{i,j}Z_{i-j} I\{|X^q_i| \leq n
    \delta\}\Big| > \frac{n\gamma}{2q+1}\Big).
\end{align*}
We replace below, for simplicity, $\gamma/(2q+1)$ with $\gamma$. Since
the above sum has a finite number of terms, it is enough to prove the
appropriate convergence to zero for each one of the terms
separately. For simplicity we consider $j=0$. Denote
\begin{align*}
  B_n & = \Bigl\{ \sup_{k \leq n}\bigl|\sum_{i=1}^k
  A_{i,0}Z_{i} I\{|X^q_i| \leq n   \delta\}\bigr| > n\gamma\Bigr\}, 
\end{align*}
so that we can write for a small $\rho>0$
\begin{align*}
  & P\Big(\sup_{k \leq n}\Big|\sum_{i=1}^k
    A_{i,0}Z_{i} I\{|X^q_i| \leq n   \delta\}\Big| > n\gamma\Big) \\
  & \quad = P\Big(B_n\cap \big\{ |Z_m|\leq n\rho \ \ \text{for all
    $m=1-q,\ldots, n+q$}\big\}\Big)\\
  & \qquad +  P\Big( B_n\cap \big\{ |Z_m|> n\rho \ \ \text{for  exactly one
    $m=1-q,\ldots, n+q$}\big\}\Big)\\
  & \qquad +  P\Big( B_n\cap \big\{ |Z_m|> n\rho \ \ \text{for two or more
    $m=1-q,\ldots, n+q$}\big\}\Big)\\
  & \quad := p_1(n)+p_2(n)+p_3(n)\,.
\end{align*}
Clearly, for every $\rho>0$,
\begin{align*}
\lim_{n\to\infty} \frac{p_3(n)}{nP(|Z|> n)} = 0.
\end{align*}
Next, select $0<\theta<\delta/(2q+1)$, and introduce the event
\begin{align*}
C_n = \Bigl\{ |A_{i_1,j}Z_{i_1-j}|>n\theta \ \ \text{for some
  $i_1=1,\ldots, n, \, |j|\leq q$}\Bigr\}.
\end{align*}
Then
\begin{align*}
p_1(n) & \leq P\Big( C_n\cap \big\{ |Z_m|\leq n\rho \ \ \text{for all
  $m=1-q,\ldots, n+q$}\big\}\Big)\\
& \quad + P\Big( B_n\cap C_n^c\cap \big\{ |Z_m|\leq n\rho \ \ \text{for all
  $m=1-q,\ldots, n+q$}\big\}\Big) \\
& := p_{11}(n) + p_{12}(n)\,.
\end{align*}
By stationarity,
\begin{align*}
 p_{11}(n) & \leq (n+2q) P\Big( |Z_0|\leq n\rho, \, \max_{|j|\leq
   q}|A_{j,j}Z_0|>n\theta\Big)\\
 &\leq (n+2q) P\Big( \max_{|j|\leq    q}|A_{j,j}|I\big\{  \max_{|j|\leq
  q}|A_{j,j}| >\theta/\rho\big\} |Z_0|>n\theta\Big)\,.
\end{align*}
Therefore,
\begin{align*}
  \lim_{n\to\infty} \frac{p_{11}(n)}{nP(|Z|> n)} & = \theta^{-\alpha}
  E\Bigl[\, \max_{|j|\leq    q}|A_{j,j}|^\alpha I\bigl\{  \max_{|j|\leq
    q}|A_{j,j}| >\theta/\rho\bigr\}\Bigr]\,,
\end{align*}
and this expression can be made arbitrarily small by selecting
$\rho$ small in comparison with $\theta$. Furthermore, the choice of
$\theta$ guarantees that, on the event $C_n^c$, one automatically
has $|X^q_i| \leq n   \delta$ for each $i=1,\ldots, n$. Therefore,
\begin{align*}
  p_{12}(n) \leq P\Big(\sup_{k \leq n}\Big|\sum_{i=1}^k
    A_{i,0}Z_{i} I\{ |Z_{i}|\leq n\rho\}\Big|>n\gamma\Big)
\end{align*}
Put $m_{n\rho} = E Z I\{|Z| \leq n\rho\}$, $\tilde S_k = \sum_{i=1}^k
A_{i,0}(Z_{i} I\{ |Z_{i}|\leq n\rho\}-m_{n\rho})$ and take $p >
\alpha$ such that $E |A_{0,0}|^p < \infty$.  Then
\begin{align}\label{eq:twoterms}
  & P\Big(\max_{k \leq n}\Big|\sum_{i=1}^k
  A_{i,0}Z_{i} I\{ |Z_{i}|\leq n\rho\}\Big|>n\gamma\Big) \nonumber \\
  & \quad \leq 
  P\Big(\max_{k \leq n}|\tilde S_k|>n\gamma/2\Big) + 
  P\Big(\max_{k \leq n} \Big|\sum_{i=1}^k A_{i,0}m_{n\rho}\Big| > n \gamma/2\Big)\,.
\end{align}

By Markov's inequality the second term in \eqref{eq:twoterms} is
bounded above by 
\begin{align*}
   P\Big(\sum_{i=1}^n |A_{i,0}| > \gamma n/2|m_{n\rho}|\Big) & \leq 
   \Big(\frac{n \gamma}{2 |m_{n \rho}|}\Big)^{-p}  E\Big(\sum_{i=1}^n
   |A_{i,0}|\Big)^p \\
   & \leq \Big(\frac{n \gamma}{2 |m_{n \rho}|}\Big)^{-p}  n^{p-1} E
   \sum_{i=1}^n |A_{i,0}|^p \\
   & = \Big(\frac{\gamma}{2 |m_{n \rho}|}\Big)^{-p} E|A_{0,0}|^p. 
\end{align*}

Since $p > \alpha > 1$, $E|A_{0,0}|^p < \infty$,  and $|m_{n\rho}| \sim C
n\rho P(|Z|> n\rho)$ it 
follows that 
\begin{align*}
    \limsup_{n \to \infty}\frac{P\Big(\sup_{k \leq n} \Big|\sum_{i=1}^k
      A_{i,0}m_{n\rho}\Big| > n \gamma/2\Big)}{nP(|Z| > n)}
    \leq \limsup_{n \to \infty} C \frac{[n \rho P(|Z| >
      n\rho)]^p}{nP(|Z|>n)} = 0.
\end{align*}

To handle the first term in \eqref{eq:twoterms} we divide into two
cases. For $1 < \alpha < 2$ we can take $\alpha < p < 2$ and use
the fact that $\tilde S_k$  is a martingale with respect to
$\mathcal{F}_k = \sigma(\{A_{i,0}\}_{i=1}^n, Z_1, \dots, Z_k)$. Then
the Burkholder-Davis-Gundy inequality implies that
\begin{align}\label{eq:secondrow}
  P\Big(\max_{k \leq n}|\tilde S_k|>n\gamma/2\Big)  &\leq 
  \frac{C}{(n\gamma)^p}E\Big([\tilde S]_n^{p/2}\Big) \nonumber \\ &\leq
  \frac{C}{(n\gamma)^p} E\Big(\sum_{i=1}^n |A_{i,0}|^p|Z_{i} I\{
  |Z_{i}|\leq n\rho\}-m_{n\rho}|^p \Big) \\
  &\leq
  \frac{C}{(n\gamma)^p} n E|A_{0,0}|^p (E|Z|^pI\{
  |Z|\leq n\rho\}+|m_{n\rho}|^p) \nonumber \\
  & \sim \frac{C}{(n\gamma)^p} n (E|A_{0,0}|^p (n\rho)^pP(|Z| > n \rho)
  +|m_{n\rho}|^p), \nonumber
\end{align}
where, in the last step, we use Karamata's theorem. 
In particular, 
\begin{align*}
  \lim_{\rho \to 0}\limsup_{n \to \infty}\frac{\frac{C}{(n\gamma)^p} n
    (E|A_{0,0}|^p (n\rho)^pP(|Z| > n \rho) 
  +|m_{n\rho}|^p)}{nP(|Z|>n)} =0. 
\end{align*}
For $\alpha \geq 2$ a variation of the  Fuk-Nagaev inequality (see
\cite{P95} 2.6.6, p.\ 
79) implies  
\begin{align*}
  P\Big(\max_{k \leq n}|\tilde S_k|>n\gamma/2\Big) 
  & = E\Big[ P\Big(\max_{k \leq n}|\tilde S_k|>n\gamma/2\Big|
  \{A_{i,0}\}\Big)\Big] \\ 
  &\leq E\Big[C_1 (n \gamma)^{-p} \sum_{i=1}^n |A_{i,0}|^p E|Z_iI\{|Z_i| \leq
  n\rho\}-m_{n\rho}|^p\Big] \\ & \quad + E\Big[\exp\Big\{-
  C_2 n^2 \Big(\sum_{i=1}^n A_{i,0}^2 \Var(ZI\{|Z| \leq n
  \rho\})\Big)^{-1}\Big\}.  
\end{align*}
The first of these terms can be bounded just as \eqref{eq:secondrow}
above.  To handle the second term we write $W_n := \sum_{i=1}^n
A_{i,0}^2$ and note that, since $\alpha\geq 2$, 
$\Var(ZI\{|Z| \leq n \rho\})$ is a slowly varying function (this
quantity 
is even bounded when $\Var(Z) < \infty$). Therefore, it is bounded by
$n^\vep$ for all $n$ sufficiently large, where we choose
$\vep$ to satisfy $\beta > \frac{\alpha-1}{1-\vep}$.  Then it follows
that for each $\lambda > 0$, 
\begin{align*}
  & E\Big[\exp\Big\{-
  C_2 n^2 \Big(W_n \Var(ZI\{|Z| \leq n
  \rho\})\Big)^{-1}\Big\} \leq 
  E\Big[\exp\Big\{-
  C_3\frac{n^{2-\vep}}{W_n}\Big\}\Big]
  \\ \quad &=  
  E\Big[\exp\Big\{- C_3\frac{n^{2-\vep}}{W_n}\Big\}
  I\{W_n \leq \lambda n^{2-\vep}/\log n\}\Big]\\ 
  & \qquad +  E\Big[\exp\{- C_3\frac{n^{2-\vep}}{W_n}\Big\}
  I\{W_n > \lambda n^{2-\vep}/\log n\Big]
  \\ \quad &\leq 
  n^{-C_3 \lambda} +  P(W_n > \lambda n^{2-\vep}/\log n).
\end{align*}
In particular we may choose $\lambda >(\alpha-1)/C_3$, which will
imply 
\begin{align*}
  \frac{n^{-C_3 \lambda}}{nP(|Z| > n)} \to 0,
\end{align*}
as $n \to \infty$. We also have, for large $n$, by the choice of
$\vep$, 
\begin{align*}
\frac{P(W_n > \lambda n^{2-\vep}/\log n)}{n P(|Z| > n)} & \leq 
\frac{n^{-\beta(2-\vep)}}{nP(|Z| > n)}E\Big(\sum_{i=1}^n
A_{i,0}^2\Big)^{\beta}\\
 & \leq
\frac{n^{-\beta(1-\vep)}}{nP(|Z| > n)}E A_{0,0}^{2\beta} \to 0,
\end{align*}
by assumption \eqref{A.additional.assumption}.

Finally, the term $p_2(n)$ can be treated in the same way as the term
$p_1(n)$, if one notices that the single large value of $Z_m$ can
contribute to at most $2q+1$ different $X_i$. If one chooses $\delta$
small enough so that $(2q+1)\delta<\gamma$, then these terms can be
excluded from the sum $\sum_{i=1}^k X^q_iI\{|X^q_i| \leq n
  \delta\}$ in the first place. Hence the statement of the lemma.
\end{proof}


\appendix

\section{Framework}

Let $\Eb$ be a locally compact complete separable metric space and
consider the space $\Np$ of Radon point measures on $\Eb$. In the
main part of the paper $\Eb$ will be the space $[0,1] \times
(\R^{d(q+1)} \setminus \{0\})$ for some $q \geq 0$ and $d\geq 1$,
but here it can be quite arbitrary. Let $(h_i)_{i\geq 1}$ be a
countable dense collection of
functions in $C_K^+(\Eb)$, the space of nonnegative continuous
functions on $\Eb$ with compact support, such that $\xi_n(h_i) \to
\xi(h_i)$ as $n\to \infty$ for each $i \geq 1$ implies $\xi_n \vague
\xi$ in $\Np$. Here $\vague$ denotes vague convergence. The existence
of such a sequence $(h_i)_{i \geq 1}$ is established by \cite{K83}
\citep[see also][Proposition 3.17]{R87}. Note also that the functions
$h_i$ may be chosen to be Lipschitz with respect to the metric on
$\Eb$. This follows from the fact that the approximating
functions in the version of the Urysohn lemma used for the purpose of
this construction are already Lipschitz \citep[see][Lemma
3.11]{R87}. In particular, a measure $\xi$ in $\Np$ is uniquely
determined by the sequence $(\xi(h_i))_{i\geq 1}$.  We may and will
assume that the collection $(h_i)_{i\geq 1}$ is closed under multiplication by
positive rational numbers.

We can identify $\Np$ with a closed subspace of
$[0,\infty)^\infty$ via the mapping $h: \Np \to
[0,\infty)^\infty$ given by  $h(\xi) = (\xi(h_i))_{i \geq 1}$. To
see that $h(\Np)$ is closed in $[0,\infty)^\infty$, let
$(x_i^n)_{i\geq 1}$ be a convergent sequence in $h(\Np)$. That is, $x_i^n \to
x_i$ for each $i$. Then there exist $\xi_n \in \Np$ such that
$\xi_n(h_i) = x_i^n$ for each $i\geq 1$. The collection
$(\xi_n)_{n\geq 1}$ is relatively compact
in $\Np$ because $\sup_n \xi_n(h_i) = \sup_n x_i^n < \infty$ for
each $i$. Hence, there is a convergent subsequence $\xi_{n_k} \to
\text{ some } \xi$. This $\xi$ necessarily satisfies $\xi(h_i) =
x_i$ and we conclude that $(x_i)_{i\geq 1} \in h(\Np)$. Thus,
$h(\Np)$ is closed.

The vague convergence on $\Np$ can be metrized via a metric $d$
induced from $[0,\infty)^\infty$, defined by
\begin{align}\label{metd}
  d(x,y) = \sum_{i=1}^\infty 2^{-i}\frac{|x_i - y_i|}{1+|x_i - y_i|},
\end{align}
for elements $x = (x_i)_{i\geq 1}$ and $y = (y_i)_{i\geq 1}$ in
$[0,\infty)^\infty$. This
makes $\Np$ into a complete separable metric space (since it is a
closed subspace of the complete separable metric space
$[0,\infty)^\infty$). The open ball of radius $r>0$
in $\Np$ centered at $\xi$ is denoted $B_{\xi,r}$. Recall that we
denote by $\xi_0$ the null measure in $\Np$.

We will consider convergence of Radon measures $m$ on the space
$\Np$. The framework considered here is that of \cite{HL06} where the
underlying space, denoted $\bfS$ by \cite{HL06}, is taken to be
$\Np$. The space of Radon measures on $\Np$ whose restriction to $\Np
\setminus B_{\xi_0,r}$ is finite for each $r > 0$ is denoted $\bfM_0 =
\bfM_0(\Np)$.  Convergence in $\bfM_0$ ($m_n \to m$) is defined as the
convergence $m_n(f) \to m(f)$ for all $f \in C_0(\Np)$, the space of
bounded continuous functions on $\Np$ that vanishes in a neighborhood
of ``the origin'' $\xi_0$.

The typical situation in this paper is that we have a sequence of
random point measures $(N_n)$ on $\Eb$, and we are interested in the
convergence
\begin{align*}
  m_n(\cdot) := r_n P(N_n \in \cdot) \to m(\cdot),\quad \text{ in } \bfM_0.
\end{align*}

\subsection{Convergence in $\bfM_0(\Np)$}

We start with  relative compactness criteria. For measures on
a general metric space such criteria are given in Theorem 2.7 in
\cite{HL06}.
\begin{thm} \label{rcnp}
  Let $M \subset \bfM_0(\Np)$. $M$ is relatively compact if
  \begin{itemize}
  \item[(i)] for each $\vep > 0$,
    \begin{align*}
      \sup_{m \in M} m\Big(\xi: \sum_{i=1}^\infty
    2^{-i}\frac{\xi(h_i)}{1+\xi(h_i)} > \vep\Big) < \infty,
  \end{align*}
  and
  \item[(ii)]  for each $h \in C_K^+(\Eb)$ and $\delta > 0$ there
    exists $R$ such that
    \begin{align*}
      \sup_{m \in M}m(\xi: \xi(h) > R) \leq
    \delta.
  \end{align*}
  \end{itemize}
\end{thm}
\begin{proof}
  We need to check (2.2) and (2.3) of Theorem 2.7 in
  \cite{HL06}. Since the metric on $\Np$ is given by \eqref{metd} (i)
  immediately implies (2.2) in that reference.

  Next note that any set of the form $\prod_{i=1}^\infty [0,R_i]$ is a
  compact subset of $[0,\infty)^\infty$. Hence, $C = \{\xi:
  \xi(h_i) \leq R_i \text{ for each } i\} \setminus
  B_{\xi_0,\vep}$ is a compact subset of $\Np \setminus B_{\xi_0,\vep}$
  and
  \begin{align*}
    \sup_{m \in M} m(\Np \setminus (B_{\xi_0,\vep} \cup C)) &\leq
    \sup_{m \in M} m(\xi:
    \xi(h_i) > R_i \text{ some } i \geq 1) \\ & \leq \sup_{m \in M}
    \sum_{i=1}^\infty m(\xi: \xi(h_i) > R_i).
  \end{align*}
  By (ii) we can take $R_i$ such that $\sup_{m\in M}m(\xi: \xi(h_i) > R_i) <
  2^{-i} \delta$, which implies (2.3) of \cite{HL06}.
\end{proof}

To show actual convergence, one needs, in addition to relative
compactness, to identify subsequential limits. For
this purpose we define for $g_1,g_2 \in C_K^+(\Eb)$,
$\vep_1,\vep_2 >
0$, a function
$F_{g_1,g_2,\vep_1,\vep_2}: \Np \to [0,\infty)$ by
\begin{align} \label{e:new.Laplace}
  F_{g_1,g_2,\vep_1,\vep_2}(\xi) =
  (1-\exp\{-(\xi(g_1)-\vep_1)_+\})
  (1-\exp\{-(\xi(g_2)-\vep_2)_+\})\,.
\end{align}
Note that each
$F_{g_1,g_2,\vep_1,\vep_2}$ is a bounded continuous function that
vanishes on a neighborhood of the null measure $\xi_0$.

\begin{lem} \label{l:equal.m}
Let $m_1$, $m_2$ be measures in $\bfM_0(\Np)$.  If
for all Lipschitz functions $g_1,g_2 \in C_K^+(\Eb)$,
$\vep_1,\vep_2 > 0$, one has $m_1(F_{g_1,g_2,\vep_1,\vep_2}) =
m_2(F_{g_1,g_2,\vep_1,\vep_2})$, then $m_1 = m_2$.
\end{lem}

\begin{proof}
We use the assumption with $g_i=h_{j_i}$, $i=1,2$. Replacing
 $h_{j_1}$ by $bh_{j_1}$ and $\vep_{1}$ by $b\vep_{1}$ with
positive rational $b$, and let
$b\to\infty$ and $\vep_2\to 0$, we obtain
\begin{align}\label{laplace}
&\int_{\Np}I\{\xi(h_{j_1})\geq \vep_1\}e^{-\xi(h_{j_2})} m_1(d\xi)
\nonumber \\
&\quad = \int_{\Np}I\{\xi(h_{j_1})\geq \vep_1\}e^{-\xi(h_{j_2})}
m_2(d\xi)\,.
\end{align}
Replacing, in \eqref{laplace}, $h_{j_2}$ by $bh_{j_2}$ as above, and
letting $b\to 0$, we obtain also
\begin{align}\label{equal}
m_1(I\{\xi(h_{j_1})\geq \vep_1\})=
m_2(I\{\xi(h_{j_1})\geq \vep_1\}).
\end{align}

Since the family $(h_i)_{i\geq 1}$ is dense in  $C_K^+(\Eb)$, we
conclude that \eqref{laplace} holds with
$h_{j_2}$ replaced by any function in  $C_K^+(\Eb)$.
To see that \eqref{equal} and \eqref{laplace} imply $m_1 = m_2$ we
define, for any $j_1 \geq 1$ and $\vep_1 > 0$, probability measures
on $\Np$ by
\begin{align*}
  \tilde m_1(\cdot) &= \frac{m_1(\cdot \cap\{\xi: \xi(h_{j_1}) \geq
    \vep_1\})}{m_1(\xi: \xi(h_{j_1}) \geq \vep_1)} \\
  \tilde m_2(\cdot) &= \frac{m_2(\cdot \cap\{\xi: \xi(h_{j_1}) \geq
    \vep_1\})}{m_2(\xi: \xi(h_{j_1}) \geq \vep_1)}.
\end{align*}
The uniqueness property of the Laplace functionals
\citep[see][Section 3.2]{R87} \eqref{laplace} implies that $\tilde
m_1$ and $\tilde m_2$
coincide. Hence $m_1$ and $m_2$ coincide on the set $\{ \xi(h_j)\geq
\vep\}$ for any $j$ and $\vep$.  Letting $\vep\to 0$ we obtain the claim.
\end{proof}

Finally, we are ready to state necessary and sufficient conditions for
convergence in $\bfM_0(\Np)$.
\begin{thm}\label{sc}
  Let $m,m_1,m_2,\dots$ be measures in $\bfM_0(\Np)$.The condition
  \begin{align*}
    \lim_{n \to \infty} m_n(F_{g_1,g_2,\vep_1,\vep_2}) =
    m(F_{g_1,g_2,\vep_1,\vep_2})
  \end{align*}
  for all $g_1,g_2 \in C_K^+(\Eb)$, $\vep_1,\vep_2 > 0$,
  is necessary and sufficient for the convergence  $m_n
  \to m$ in $\bfM_0(\Np)$. Furthermore, it is sufficient to check the
  condition only for the Lipschitz functions in $C_K^+(\Eb)$.
\end{thm}
\begin{proof}
The necessity of the condition is obvious. For the sufficiency
we start with checking that the sequence $(m_n)_{n\geq 1}$ is relatively
compact in $\bfM_0(\Np)$, for which we will check (i) and (ii) in Theorem
\ref{rcnp}.

Start by choosing a Lipschitz collection $(h_i)_{i \geq 1}$ as
above. Take $\vep > 0$. With
$J_\vep = \lceil -\log_2\vep\rceil +1$ we have
\begin{align*}
  m_n & \Big(\Big\{\xi: \sum_{i=1}^\infty 2^{-i}
      \frac{\xi(h_i)}{1+\xi(h_i)}\geq \vep\Big\}\Big)
      \leq m_n \Big(\Big\{ \xi:\, \sum_{i=1}^{J_\vep} 2^{-i}
 \frac{\xi(h_i)}{1+\xi(h_i)} > \frac{\vep}{3}\Big\}\Big) \\
& \leq \sum_{i=1}^{J_\vep} m_n \Big(\Big\{ \xi:\,
\frac{\xi(h_i)}{1+\xi(h_i)} > \frac{\vep}{3 J_\vep}\Big\}\Big)\\
& \leq \sum_{i=1}^{J_\vep} m_n \Big(\Big\{ \xi:\, \xi(h_i)
> \frac{1}{(3 J_\vep/\vep-1)}\Big\}\Big)\,.
\end{align*}
Note that, for any $h \in C_K^+(\Eb)$ and $R > 0$,
\begin{align} \label{e:R}
  m(F_{h,h,R/2,R/2}) &= \int (1-e^{-(\xi(h)-R/2)_+})^2 m(d\xi) \\ & \geq
  \int (1-e^{-(\xi(h)-R/2)_+})^2I\{\xi(h) > R\} m(d\xi) \\ &\geq
  (1-e^{-R/2})^2m(\xi: \xi(h) > R).
\end{align}

For $\vep>0$ we choose $R=R(\vep)= \frac{2}{3 J_\vep/\vep-1}$.
By the assumption of the proposition  there is
$n_1$ such that for
all $n\geq n_1$ the bound
$$
m_n(F_{h_i,h_i,R/2,R/2}) \leq
m(F_{h_i,h_i,R/2,R/2})+1
$$
holds for each $i=1,\ldots,
J_\vep$. It follows from \eqref{e:R} that for all such $n$,
\begin{align*}
& m_n \Big(\Big\{ \xi \in \Np:\, \sum_{i=1}^\infty 2^{-i}
\frac{\xi(h_i)}{1+\xi(h_i)} \geq \vep\Big\}\Big) \\
& \leq \big( 1-e^{-R(\vep)}\big)^{-2} \sum_{i=1}^{J_\vep}[
m(F_{h_i,h_i,R(\vep)/2,R(\vep)/2})+1]\,,
\end{align*}
which is finite, establishing (i) in Theorem \ref{rcnp}.

The next step is to check (ii) in Theorem \ref{rcnp}. For $h\in
C_K^+(\Eb)$ and $R > 0$ we have by \eqref{e:R}
\begin{align*}
  \limsup_{n\to\infty} m_n (\{ \xi:\,
      \xi(h) > R \})
  & \leq \limsup_{n\to\infty} \big( 1-e^{-R/2}\big)^{-2}
  m(F_{h,h,R/2,R/2}) \\
& = \big( 1-e^{-R/2}\big)^{-2} m(F_{h,h,R/2,R/2}) \,.
\end{align*}
The latter expression converges to
zero as $R \to\infty$, which implies (ii) in Theorem \ref{rcnp}.

We conclude that $(m_n)$ is relatively compact in $\bfM_0(\Np)$.

Since the assumptions of Lemma \ref{l:equal.m} are satisfied for any
subsequential vague limit point of the sequence $(m_n)$ and the measure $m$, we
conclude that all subsequential vague limit points of the sequence
$(m_n)$ coincide with $m$ and, hence, $m_n \to m$ in $\bfM_0(\Np)$.
\end{proof}

\subsection*{A mapping theorem}

The general version of the mapping theorem is given in Theorem
2.5 in \cite{HL06}. Here we will state a useful special case.

\begin{lem} \label{mapthm}
  Suppose $m_n \to m$ in $\bfM_0(\Np)$ and $f: \Eb \to \R^d$ is a
  measurable function with a bounded support, such that $m\bigl(\xi:\,
  \xi({\mathcal D}_f)>0\bigr)=0$, where ${\mathcal D}_f$ is the set of
  discontinuities of the function $f$. Define $T: \Np
  \to \R^d$, by $T(\xi) = \xi(f)$. Then
  \begin{align*}
    m_n \circ T^{-1}(\cdot) \to m \circ
    T^{-1}(\cdot),
  \end{align*}
  in $\bfM_0(\R^d)$.
\end{lem}

\begin{proof}
  This follows from Theorem 2.5 in \cite{HL06} since $T$ is
  discontinuous  on a set of measure $m$ equal to zero, $T(\xi_0) =
  0$, and $T$ is continuous at $\xi_0$.
\end{proof}

\end{document}